\documentclass[10pt, reqno]{amsart}
\usepackage{amsmath}

\usepackage[backref=page]{hyperref}

\usepackage{amsthm}
\usepackage{amssymb}

\usepackage{physics}

\newtheorem{thm}{Theorem}[section]
\newtheorem{cor}[thm]{Corollary}
\newtheorem{lem}[thm]{Lemma}
\newtheorem{prop}[thm]{Proposition}
\theoremstyle{definition}
\newtheorem{defn}[thm]{Definition}

\newtheorem{rem}[thm]{Remark}
\numberwithin{equation}{section}

\newcommand{\Holder}[0]{H\"older}
\newcommand{\Schr}[0]{Schr\"odinger}

\begin{document}
%

\author[Xiaoyan Su]{Xiaoyan Su}
\address{Institute of Applied Physics and Computational Mathematics, 100094, P.R.China}
\email{suxiaoyan0427@qq.com}

\author[Jiqiang Zheng]{Jiqiang Zheng}
\address{Institute of Applied Physics and Computational Mathematics, 100094, P.R.China}
\email{zhengjiqiang@gmail.com}

\title{H\"older regularity for the time fractional Schr\"odinger equation}

\begin{abstract}
In this paper, we investigate the H\"older regularity of solutions to the time fractional Schr\"odinger equation of order $1<\alpha<2$, which interpolates between the Schr\"odinger and wave equations. This is inspired by Hirata and Miao's work \cite{Hirata} which studied the fractional diffusion-wave equation.  First, we give the asymptotic behavior for the oscillatory distributional kernels and their Bessel potentials by using Fourier analytic techniques. Then, the space regularity is derived by employing some results on singular Fourier multipliers. Using the asymptotic behavior for the above kernels, we prove the time regularity. Finally, we use mismatch estimates to prove the pointwise convergence to the initial data in H\"older spaces. In addition, we also prove H\"older regularity result for the Schr\"odinger equation.

\end{abstract}
\maketitle
\begin{center}
 \begin{minipage}{125mm}
 \end{minipage}
\end{center}


\section{Introduction}

In this paper, we investigate the \Holder{} regularity of solutions to the following Cauchy problem,
\begin{equation}\label{fde}
\left \{\begin{split}
  &i^\alpha D_t^\alpha u(t,x) = -\Delta u(t,x)+f(t,x), &(t,x)&\in {\mathbb R^+}\times \mathbb R^n,\\
 & u(0,x)=u_0(x), \   \partial_t u(0, x)=u_1(x), & x&\in \mathbb R^n, \end{split}
\right.
 \end{equation}
where $D_t^\alpha$ is the Caputo derivative of order $1<\alpha<2$, which interpolates between the \Schr{} equation ($\alpha=1$) and wave equation ($\alpha=2$).

Time fractional differential equations have recently become a topic of active research.
Many of these papers study the time fractional generalization of the heat equation or the wave equation, with form
\begin{align}\label{dwe}
D_t^\alpha u(t,x) = \Delta u(t,x)+f(t,x), \quad (t,x)\in {\mathbb R^+}\times \mathbb R^n,
\end{align}
which we will refer as `diffusion-wave' equations.
For example, the fundamental solutions of the diffusion-wave equations have been given in \cite{ Mainardi, Schneider} and the asymptotic expansions of the fundamental solutions have been fully developed in \cite{Allen, Kim2}. The $L^p_t L^q_x$ theory for time fractional evolution equations with variable coefficients has been established in \cite{Kim}.
 Some probabilistic methods have been used in \cite{Fujita, Fujita2} to study these diffusion-wave equations. Also, see \cite{Hirata, Miao2005, Miao} for more topics.

The time fractional \Schr{} equation has been studied in \cite{ Dong,Grande, Laskin 2, Naber, Su, Su2}, but our equation is different because they restrict the order $\alpha$ of the time fractional derivative to the range $(0,1)$. In contrast, we consider here the time fractional derivative $D_t^\alpha$ of order $\alpha \in (1,2)$, which has not been considered in the previously mentioned papers. In addition, in our equation \eqref{fde}, the term $D_t^\alpha u$ is multiplied by $i^\alpha$,  which makes it different from the diffusion-wave equation \eqref{dwe}. Also, when $\alpha = 1$, the equation \eqref{fde} is the classical \Schr{} equation, and when $\alpha = 2$, the equation \eqref{fde} is the wave equation.

Using the Fourier-Laplace transform, it is easy to solve the Cauchy problem  \eqref{fde} in the sense of mild solutions (see Definition \ref{mildsolution} below). In this paper, we are interested in the minimal regularity required to understand \eqref{fde} in a pointwise sense. Intuitively, all the terms in the equation make sense pointwise as long as $u$ is $C_t^\alpha$ in time (see Definition \ref{time holder space} below) and $C^2$ in space. So the question becomes: what regularity should we impose on the initial data $(u_0,u_1)$ so that the mild solution is a classical solution in the sense of Definition \ref{classical-solution}?

   To answer this question, we can first look at the endpoint cases $\alpha \in \{1,2\}$. For the wave equation ($\alpha=2$), it is well known that in order to make sure that the solution is classical ($C^2$ in both space and time), the initial data should belong to $C^{[n/2] + 2}(\mathbb R^n)$. So heuristically, the solution loses $[n/2]$ `derivatives'. This result can be obtained either by giving an explicit formula for the solution (see Evan's \cite{Evans} or Miao's \cite{Miao}) or by investigating the order of the fundamental solution as a distribution, as in H\"ormander's book \cite{Homander}.

For the \Schr{} equation ($\alpha=1$), to the best of the authors' knowledge, classical solutions of this H\"older regularity has not been studied, since it is more natural from physical considerations to study this equation with $L^2$ based Sobolev spaces. However, we can still ask similar question for the Schr\"odinger equation: what is the minimal regularity can guarantee the existence of the classical solution?  To answer this question, we also prove a new result for the classical \Schr{} equation (see Theorem \ref{main-result-2} below).

 The existence of classical solutions for the diffusion-wave equation \eqref{dwe} is proven by studying the asymptotic behavior of the fundamental solution using Fox $H$-functions, see \cite{Kemppainen, Kochubei, Kochubei2, Pskhu}. However, we cannot adapt this method to study \eqref{fde} because in our case, the variable $z = i^{-\alpha} x t^{-\alpha/2}$ lies in a critical line $|\arg z |= \frac{\alpha \pi }2$ of the corresponding $H$-functions and the asymptotic behavior on this line cannot be given simply by analyzing the residues of $H$-functions.

 Instead, we will use Fourier analysis techniques to study \eqref{fde} for $\alpha \in (1,2)$. In contrast to the fundamental solutions for the diffusion-wave equations which have polynomial decay as $|z|\to \infty$, we prove (see Lemma \ref{asymptotic-behavior-of-kernel} below) that the distributional kernels for \eqref{fde} are oscillating without decay as $|z|\to \infty$. This motivates the study of Fourier multipliers related to \eqref{fde} in \Holder{} spaces, which allow us to prove the space regularity of solutions. To obtain the time regularity, we prove some upper bounds for the Bessel potentials of the corresponding kernels to verify the exchange of limits and integrals.

 Finally, we prove that the initial data is attained in a pointwise sense. As it is well known, this is easy to prove for the wave equation since the fundamental solution is compactly supported, a property also known as the finite speed of propagation. For the \Schr{} equation, finite speed of propagation is not available which complicates the proof. As a replacement, we use `almost finite speed of propagation', also known as `mismatch estimates' in the literature (see \cite{Killip}). For \eqref{fde} with $\alpha\in(1,2)$, it turns out that analogous mismatch estimates are also true, which allow us to prove that if the initial data belongs to a certain \Holder{} space, then solutions to \eqref{fde} with $\alpha\in(1,2)$ pointwise attain the initial data.

 Our main results are as follows:
 \begin{thm}\label{main-result-1} For $1<\alpha<2$ and any $\epsilon>0$, if $u_0 \in C^{2+\frac{n}{\alpha}+\epsilon}(\mathbb R^n)$, $u_1 \in C^{2+\frac{n}{\alpha}-\frac{2}{\alpha} +\epsilon}(\mathbb R^n)$ and $f\in C_t^{0}((0,\infty); C^{\frac{n}{\alpha}+\frac{2}{\alpha}+\epsilon}(\mathbb R^n))$, then the mild solution $u(t,x)$ defined by \eqref{mildsolution} is a classical solution defined in \eqref{classical-solution}.	
\end{thm}

For equation \eqref{fde} with $\alpha=1$, which is the Schr\"odinger equation, we only need to prescribe $u_0$, that is
\begin{equation}\label{se}
 \left \{\begin{split}
  &i \partial_t u(t,x) = -\Delta u(t,x)+f(t,x), &(t,x)&\in \mathbb R^+\times \mathbb R^n,\\
 & u(0,x)=u_0(x),  & x&\in \mathbb R^n,
 \end{split}
 \right.
 \end{equation}
and we have the following result:
\begin{thm} \label{main-result-2}
 For $\alpha=1$ and any $\epsilon>0$,  if $u_0(x)\in C^{n+2+\epsilon}(\mathbb R^n)$ and $f\in C_t^{0}((0,\infty);C^{n+\epsilon}(\mathbb R^n))$, then the mild solution defined in \eqref{sms} is a classical solution .
\end{thm}

\begin{rem}
If we ignore the small $\epsilon$, the Theorem \ref{main-result-1} and \ref{main-result-2} are sharp from the perspective of Fourier multiplier results. However, using the methods in this paper, we need a small extra $\epsilon$ to guarantee the existence of classical solution. Whether we can remove this $\epsilon$ is unclear at this moment.	
\end{rem}

\begin{rem} Although the problem in this paper is inspired by the wave equation, the methods used here cannot be applied to recover the known results for the wave equation. The main reason is that the Fourier multiplier $e^{i|\xi|^{\frac{2}{\alpha}}}$ for $\alpha=2$  behaves quite differently with the multiplier $1<\alpha< 2$. The corresponding kernel for the former has singularity at $|x|=1$, whereas the singular points for the latter appears only at $|x|=0$ or $|x|=\infty$, \cite{Miyachi} for the details.
\end{rem}

This paper is organized as follows. Some preliminaries are given in section 2. In section 3, we derive the asymptotic behavior for the oscillatory kernel functions, and some delicate upper bounds for their Bessel potentials are obtained by using Fourier analysis methods. In subsection 4.1, we explore the Fourier multiplier properties for the oscillatory Mittag--Leffler functions and prove the space regularity in H\"older spaces. In subsection 4.2, the time regularity is proved using the estimates for the Bessel potentials of the corresponding kernels. In subsection 4.3, we prove the pointwise convergence to the initial data by using the almost finite speed propagation property for the solutions. The Schr\"odinger equation will be discussed in section 5.

 \section{Preliminaries}

In this section, we will review some basic definitions and properties about the fractional derivatives, Mittag--Leffler functions, Holder spaces and Besov spaces and closely related Bessel potentials. Some results on singular Fourier multipliers  are also included. Also we recall one version Fa\'{a} di Bruno's formula which will be used in our paper later. Finally, the classical solutions and mild solutions are given using the notations introduced above.

 We denote by $\mathcal S$ the Schwartz space on $\mathbb R^n$ and by  $\mathcal S'$ of temper distributions. The Fourier transform and the inverse Fourier transform are defined by
 \[\mathcal F(f)(\xi)=\widehat f(\xi)=\frac{1}{(2\pi)^{n/2}}\int_{\mathbb R^n} e^{-i \xi \cdot x} f(x)\dd{x},\]
 and \[\mathcal F^{-1}(f)(x)=\frac{1}{(2\pi)^{n/2}}\int_{\mathbb R^n} e^{i x \cdot \xi} f(\xi)\dd{\xi}.\]
 For a function $m(\xi)$, we can associate it with an operator $T_m$ given by
 \[T_m f(x)=\frac{1}{(2\pi)^{n/2}}\int_{\mathbb R^n} e^{i x \cdot \xi} m(\xi)\widehat f(\xi)\dd{\xi}.\]
We say that `$m(\xi)$ is a multiplier from $X$ to $Y$' if $T_m$ is a bounded operator from $X$ to $Y$.

 \subsection{Fractional derivatives and properties}
 In this subsection, we review some definitions and properties about fractional derivatives, see \cite{Podlubny}.

For $\alpha> 0$, we denote $g_\alpha(t)$ for
\begin{equation}\label{g function}
	g_\alpha(t):=\begin{cases} \frac{t^{\alpha-1}}{\Gamma(\alpha)},\quad &t>0,\\ 0,\quad &t\leq 0,
	\end{cases}
\end{equation}
where $\Gamma(\alpha)$ is Gamma function. When $\alpha=0$, we denote $g_0(t)=\delta(t)$.
And the Riemann-Liouville fractional integral of oder $\alpha\geq 0$ is defined as follows:
\begin{align*}
	I_t^\alpha f(t):=(g_\alpha *f)(t).
\end{align*}
Let us use $\lceil \alpha \rceil$ to denote the smallest integer greater than or equal to $\alpha$.
Then the  Riemann-Liouville fractional derivative of order $\alpha>0$ is given by
\begin{equation*}
	\partial_t^\alpha f(t)= D_t^{\lceil \alpha \rceil}I_t^{\lceil \alpha \rceil-\alpha}f(t),
\end{equation*}
where $D_t^m =\frac{d^m}{dt^m}$ is the classical $m$th derivative.

Generally, the Caputo derivative of order $\alpha>0$ is given by
\begin{equation*}
	D_t^\alpha f(t)=I_t^{\lceil \alpha \rceil-\alpha} D_t^{\lceil \alpha \rceil}f(t).
\end{equation*}
From the above definitions, it is obvious that we need $f(t)$ to be $C^{\lceil \alpha \rceil}$-times differentiable to define the Caputo derivatives. Whereas, less regularity is required to define its Riemann-Liouville derivatives.

 In addition, it is easy to check that the following equation is valid for smooth enough function $f(t)$ using integration by parts,
\begin{equation}\label{RL-C relationship}
	D_t^\alpha f(t)=\partial_t^\alpha \big(f(t)-\sum_{k=0}^{\lceil \alpha \rceil-1}f^{(k})(0) g_{k+1}(t)\big).
\end{equation}
Therefore, we can redefine the Caputo derivative using Riemann-Liouville fractional derivative according to \eqref{RL-C relationship} to lower the regularity for $f(t)$.

Applying the properties of the Laplace transform, we have
\begin{align*}
\mathcal L(\partial_t^\alpha f(t))(\lambda)&= \lambda ^\alpha \mathcal L(f(t))-\sum_{k=0}^{\lceil \alpha \rceil-1}(g_{\lceil \alpha \rceil-\alpha}*f)^{(k)}(0)\lambda^{\lceil \alpha \rceil-1-k},\\
\mathcal L(D_t^\alpha f(t))(\lambda)&= \lambda ^\alpha \mathcal L(f(t))-\sum_{k=0}^{\lceil \alpha \rceil-1}f^{(k)}(0)\lambda^{\lceil \alpha \rceil-1-k}.	
\end{align*}
where $\mathcal L$ denotes the Laplace transform. As we can see from the above equalities, the Laplace transform for the Caputo derivative is simpler than the Riemann-Liouville derivative, that is part of the reason why we choose the Caputo fractional derivatives in our paper.

To be more clear, we give the definition $D_t^\alpha$ for $1<\alpha <2$ in the following:
 \begin{defn}
Let $1<\alpha<2$, using \eqref{RL-C relationship} we define the fractional Caputo derivative of a function $f(t)$ by
 \begin{align*}
 	D_t^\alpha f(t)=D^2_{t}\left(\frac{1}{\Gamma(2-\alpha)}\int_0^t (t-\tau)^{1-\alpha}\big(f(\tau)-f(0)-f'(0)\tau\big)\dd{\tau}\right),
 \end{align*}
 	
 \end{defn}

 For the above definition, we have
\begin{align}
	\mathcal L(D_t^\alpha f(t))(\lambda)= \lambda ^\alpha \mathcal L(f(t))-\lambda ^{\alpha-1}f(0)-\lambda ^{\alpha-2}f'(0).
\end{align}

Now we introduce time H\"older space which closely related to the fractional Caputo derivatives above, and it is different with the spatial H\"older spaces which will be introduced later.
\begin{defn}\label{time holder space}  We use $C_t^0(0,\infty)$ to denote the set of all the continuous and uniformly bounded functions with respect to $t>0$.  For $1<\alpha<2$, we say a function $f(t)\in C_t^{\alpha}(0,\infty)$ if $g_{2-\alpha}*f(t) \in C^2{(0,\infty)}$.
 \end{defn}

\subsection{Mittag--Leffler functions}
In this subsection, we introduce the Mittag--Leffler functions, its asymptotic behaviors, and its derivatives. More information about Mittag--Leffler functions can be found in \cite{Gorenflo2014}.
\begin{defn} \cite[p.56]{Gorenflo2014}
	The two parameter Mittag--Leffler functions are given by
\begin{align}\label{definition-of-ml}
	E _ { \alpha , \beta } ( z ) = \sum _ { k = 0 } ^ { \infty } \frac { z ^ { k } } { \Gamma ( \alpha k + \beta ) } ,\quad ( \alpha > 0 , \beta \in \mathbb { C } ).
	\end{align}
\end{defn}

The most interesting properties of the Mittag--Leffler functions are associated with their Laplace transforms,
\[\mathcal L(t^{\beta-1}E_{\alpha, \beta}(\omega t^{\alpha} )(\lambda )=\int_{0}^\infty e^{-\lambda t}t^{\beta-1}E_{\alpha, \beta}(\omega t^{\alpha} )\dd{t}=\frac{\lambda^{\alpha-\beta}}{\lambda^\alpha-\omega} .\]

In the following, we give the asymptotic behavior of Mittag--Leffler functions.
\begin{thm}\cite[p.64 Theorem 4.3]{Gorenflo2014}
\label{thm-asymptotic}
 For $0<\alpha<2$, $\beta\in \mathbb C$, $m \in \mathbb N$, the following asymptotic formulas hold:
 \begin{enumerate}
 	\item  For any $z \in \mathbb C $, we have
 	\begin{align}\label{asymptotic-of-ml-zero}
 	E _ { \alpha , \beta } ( z )=\sum_{k=0}^{m}\frac{z^k}{\Gamma(\alpha k+\beta)}+O(|z|^{m+1}),\quad |z|\to 0. \end{align}
 	\item If $|\arg z|<\min \{\pi, \alpha \pi\} $, then
\begin{align}\label{asymptotic-of-ml-infity}	
E _ { \alpha , \beta } ( z ) = \frac { 1 } { \alpha } z ^ { ( 1 - \beta ) / \alpha } \mathrm { e } ^ { z ^ { 1 / \alpha } } - \sum _ { k = 1 } ^ { m } \frac { z ^ { - k } } { \Gamma ( \beta - k \alpha ) } + O \left( | z | ^ { - m - 1 } \right) ,\quad | z | \rightarrow \infty.\end{align}	
 \end{enumerate}
In particular, when $|z|\geq 1$, with $|\arg z |=\frac{\pi}{2}$, then we have
\begin{align}\label{boundedness-of-mf}
	|z^{\beta-1}E_{\alpha,\beta}(z^\alpha)|\leq C.
\end{align}
\end{thm}

The next theorem deals with the derivatives of the Mittag--Leffler functions.
\begin{thm} \cite[p.58]{Gorenflo2014} Let $m\geq 1$ be an integer, then \begin{align}\label{derivative-formula}
		 \frac { d^m } {  d  z ^m}   \left[ z ^ { \beta - 1 } E _ { \alpha , \beta } \left( z ^ { \alpha } \right) \right] = z ^ { \beta - m - 1 } E _ { \alpha , \beta - m } \left( z ^ { \alpha } \right), \quad ( m \geq 1 ).
	\end{align}

 	\end{thm}

 \subsection{H\"older spaces and Besov spaces}
 In this subsection, we will recall the H\"older spaces and Besov spaces in $\mathbb R^n$, see \cite{Tribel1983}.

  Let $\gamma=(\gamma_1, \cdots, \gamma_n)$ be a multi-index of non-negative integers, and $|\gamma |=\gamma_1+\cdots+\gamma_n$, then denote \[D^\gamma f(x)=\frac{\partial^{|\gamma|}f(x)}{\partial x_1^{\gamma_1}\cdots\partial x_n^{\gamma_n} }. \]

 \begin{defn}(\cite[p.36]{Tribel1983}\label{holder space} {\bf{ H\"older spaces $C^{s}(\mathbb R^n)$}}).
 We use $ C^0(\mathbb R^n)= C(\mathbb R^n)$ to denote the set of all complex-valued bounded and uniformly continuous functions on $\mathbb R^n$.
When $s\geq 0$ is an integer,
\begin{align*}
 C^s(\mathbb R^n)&=\{f:\mathbb R^n \to \mathbb C \mid D^{\gamma}f\in C(\mathbb R^n), \quad \text{for all } |\gamma|\leq s\}, \  \text{with}\\
 \|f\|_{C^s(\mathbb R^n)} &= \sum_{|\gamma|\le s } \|D^\gamma f\|_{C(\mathbb R^n)}.
\end{align*}
If $s>0$ is not an integer, then we write $s=[s]+\{s\}$ with $[s]$ integer, and $0< \{s\}<1$, then
\begin{align*}
 C^s(\mathbb R^n)&=\{f:\mathbb R^n \to \mathbb C \mid f\in C^{[s]}(\mathbb R^n), \  \|f\|_{C^{s}(\mathbb R^n)}<\infty\},\ \text{with}\\
 \|f\|_{C^{s}(\mathbb R^n)}&=\|f\|_{C^{[s]}(\mathbb R^n)}+\sum_{|\gamma|=[s]}\sup_{x\neq y }\frac{|D^\gamma f(x)-D^\gamma f(y)|}{|x-y|^{\{s\}}}.
\end{align*}
\end{defn}
 In this paper, we use $C^0_t((0,\infty); C^\theta(\mathbb R^n))$ to denote all the functions $f(t,x)$ such that $\|f(t,x)\|_{C^\theta(\mathbb R^n)}\in C^0_t(0,\infty)$.

\begin{defn}(\cite[p.45]{Tribel1983} {\bf{Inhomogeneous Besov space}}) Let $\{\psi_j\}_{j\geq 0}\subset \mathcal S(\mathbb R^n)$ be a family of functions such that
\[
\begin{cases}
  \operatorname{supp} \psi_0 \subset \{x||x|\leq 2\},\\
    \operatorname{supp} \psi_j \subset \{x|2^{j-1}\leq |x|\leq 2^j\},\  j\geq 1,
\end{cases}\]
for every multi-index $\gamma$ there exists a positive number $c_\gamma$ such that $2^{j|\gamma|}|D^{\gamma}\psi_j(x)|\leq c_\gamma$ for all $x\in \mathbb R^n$ and $\sum\limits_{j=0}^{\infty}\psi_j(x)=1$ for every $x\in \mathbb R^n$.

The inhomogeneous Besov space is given by
\[{B}_{p,q}^{s}(\mathbb R^n)=\{f \in \mathcal{S}'(\mathbb R^n): \|f\|_{{B}_{p,q}^{s}(\mathbb R^n)}<\infty\},\]
where
\[\|f\|_{{B}_{p,q}^s}:=
 	\begin{cases}
 		\Big(\sum\limits_{j=0}^{\infty}(2^{js}\|(\psi_j \hat {f})^{\vee}\|_{L^p})^q\Big)^{1/q},& \quad 1\leq q <\infty,\\
		\sup\limits_{j\geq 0} 2^{js} \|(\psi_j \hat {f})^{\vee}\|_{L^p}, &\quad q=\infty.
	\end{cases}\]	
	\end{defn}

According to the Lemma 1 in \cite{Taibleson1965} we have the following `Young's Theorem' in H\"older spaces given as:
\begin{lem} \label{Besov multiplier}

If $f\in C^s(\mathbb R^n)$ and $ g\in B_{1,\infty}^0(\mathbb R^n)$, then
$f*g \in C^s(\mathbb R^n)$.
	
\end{lem}

\subsection{Bessel potentials}
In this subsection, we introduce the Bessel potentials and its operator properties acting on the H\"older space.

\begin{defn}[See \cite{Calderon}]
For any $\theta \in \mathbb R$, the Bessel potential of order $\theta$  for a tempered distribution $f$ is defined by
\[J^{\theta}f(x)=\mathcal F^{-1}((1+|\xi|^2)^{-\theta/2}\widehat f(\xi))(x). \]
For $\theta>0$, the function $(1+|\xi|^2)^{-\theta/2}$ is the Fourier transform of a integrable function $G^\theta(x)$ where

\begin{equation}\label{bessel-function}
G^\theta(x) =\frac{1}{\Gamma(\theta /2)} \int_{0}^\infty e^{-|x|^2/s}e^{-s/4} s^{(-n+\theta)/2-1}\frac{\dd{s}}{s}.\end{equation}
\end{defn}

From a straightforward examination of the integral \eqref{bessel-function}, we have for some $c>0$,
\begin{align}
	G^\theta(x)&=O(e^{-c|x|}),  & &\text{as} \quad |x|\to \infty, \quad \text{and}\\
	G^\theta(x)&=\frac{|x|^{-n+\theta}}{\gamma(\theta)}+o({|x|^{-n+\theta}}),  &&\text{as} \quad |x|\to 0,
\end{align}
 if $0<\alpha <n$.

The Bessel potential acts like a smoothing operator in H\"older space, in the following sense:
\begin{prop} \cite[Theorem 6]{Taibleson1964} \label{Bessel map}
 For any $\theta \in \mathbb R$, $J^\theta$ maps $C^{s}(\mathbb R^n)$ isomorphically onto $C^{s+\theta}(\mathbb R^n)$.	
\end{prop}

\subsection{Singular multipliers}
In this subsection, we recall some known results on singular Fourier multipliers (see \cite{Miyachi} for more details).

  Let us denote by
 \[m_{a,b}^{\pm}(\xi)=\phi(\xi)|\xi|^{-b}\exp(\pm i|\xi|^a), \ \xi\in \mathbb R^n, \ a>0, \ b\in \mathbb R,\]
 where $\phi$ is a smooth function which vanishes in a neighborhood of the origin and equal to $1$ outside a compact set. For this multiplier, we denote its kernel by $K_{a,b}^{\pm}(x)=\mathcal F^{-1}( m_{a,b}^{\pm}(\xi))$.  According to \cite{Miyachi}, we have the following property:
 \begin{thm}\cite[Proposition 5.1]{Miyachi}
 \label{kernel-asymptotic}
  When $a>1$ and $b\in \mathbb R$, $K^{+}_{a, b}$ is smooth throughout $\mathbb R^n$, and
 \begin{align*}
  K_{a,b}^{+}(x) = A|x|^{\frac{b-n+na/2}{1-a}}\exp(iB|x|^{-a/(1-a)})+o(|x|^{\frac{b-n+na/2}{1-a}}), \quad\text {as}\quad |x|\to \infty,
 \end{align*}
 where $A, B$ are constants which depend on $a, b$ and $n$.	
 \end{thm}
 It is easy to notice that $K_{a,b}^{+}(x)$ and  $K_{a,b}^{-}(x)$ are are complex conjugates of each other, thus they have the same asymptotic behavior.

 In this paper, we will use the following inhomogeneous multiplier $(1+|\xi|^2)^{-b/2}$ instead of homogeneous one $|\xi|^{-b}$, and we give
 the following corollary.
\begin{cor}\label{equivalence of Riesz and Potential}
 	If we denote by
 	 \[\tilde m_{a,b}^{\pm}(\xi)=\phi(\xi)(1+|\xi|^2)^{-b/2}\exp(\pm i|\xi|^a), \ \xi\in \mathbb R^n,\  a>0, \  b\in \mathbb R,\]
 	and $\tilde K_{a,b}^{\pm}(x)=\mathcal F^{-1}(\tilde m_{a,b}^{\pm}(\xi))$, then
 	$\tilde K_{a,b}^{\pm}(x)\sim K_{a,b}^{\pm}(x)$. To be precise, when $a>1$ and $b\in \mathbb R$, $ \tilde K_{a, b}^{\pm}(x)$ is smooth throughout $\mathbb R^n$, and
 \begin{align*}
  \tilde K_{a,b}^{\pm}(x) = A|x|^{\frac{b-n+na/2}{1-a}}\exp(\pm iB|x|^{-a/(1-a)})+o(|x|^{\frac{b-n+na/2}{1-a}}), \quad\text {as}\quad |x|\to \infty,
 \end{align*}
 where $A, B$ are constants which depend on $a, b$ and $n$.
 \end{cor}
 \begin{proof}
 	
 Since we can write $\tilde m_{a,b}^{\pm}=\frac{|\xi|^b}{(1+|\xi|^2)^{b/2}} m_{a, b}^{\pm}$. Then by Lemma 2, Chapter 4.3 in \cite{Stein},  $K_b(x):=\mathcal F^{-1}\left(\frac{|\xi|^b}{(1+|\xi|^2)^{b/2}}\right)$ is a finite measure. Therefore, the property is valid.
  \end{proof}


 Let $X$ and $Y$ be spaces of functions or distributions on $\mathbb R^n$, and define $\mathcal M(X,Y)$ to be the spaces of all $m\in \mathcal S'$ such that $\|m\|_{\mathcal M(X,Y)}< \infty$, where
 \[\|m\|_{\mathcal M(X,Y)}=\sup \left\{\frac{\|\mathcal F^{-1}(m\widehat f)\|_{Y}}{\|f\|_{X}}: f\in \mathcal D \cap X, \|f\|_{X}\neq 0\right\}.\]

\begin{thm}\cite [Theorem 5.1]{Miyachi}\label{multiplier-thm}
 Let $a>0$, $\theta \in \mathbb R$. Suppose that $A\geq 1$, $m$
is a function of class $C^{[\frac{n}{2}]+1}$ on $\mathbb R^n$, $m(\xi)=0$ for $|\xi|\leq 2$ and
  \[\left| D ^ { \gamma } m ( \xi ) \right| \leq | \xi | ^ { - \theta- n a / 2 } \left( A | \xi | ^ { a - 1 } \right) ^ { |\gamma | } \quad \text{ for } | \gamma | \leqq [ n / 2 ] + 1.\]
  Then $m \in \mathcal M(C^s(\mathbb R^n), C^{s+\theta}(\mathbb R^n))$, and
  \[\|m\|_{\mathcal M(C^s(\mathbb R^n), C^{s+\theta}(\mathbb R^n))}\leq C A^{\frac{n}{2}},\]
  where the constant $C$ depends only on $a, s, \theta$ and $n$.
\end{thm}

%
  \subsection{Fa\'{a} di Bruno's formula}	
Let us also recall Fa\'{a} di Bruno's formula for later use.
 In the following lemma, we use bold symbols to distinguish multi-index integers from integers.

	\begin{lem}	 \cite[2.10. Corollary]{Constantine}\label{faa di bruno}
	Assume that $x\in \mathbb R^n$ and $y \in \mathbb R$.  Let $|\gamma|\geq 1$, and $h(x_1, x_2, \cdots, x_n)=f(g(x_1, x_2, \cdots, x_n)) $ with $g\in C^{|\gamma|}(\mathbb R^n)$ and $f\in C^{|\gamma|}(\mathbb R)$, let $y=g(x)$. There holds:
	
	\[ D^{\boldsymbol{\gamma}}h  = \sum _ { r = 1 } ^ { |\gamma| } \frac{d^r f}{dy^r} \sum _ { p (   \boldsymbol{\gamma} , r ) } (  \boldsymbol{\gamma} ! ) \prod _ { j = 1 } ^ {|\gamma| } \frac { \left[ D^{\boldsymbol{\ell_j} } g  \right] ^ { k _ { j } } } { \left( k _ { j } ! \right) \left(\boldsymbol{ \ell} _ { j } ! \right) ^ { k _ { j } } }\]
	where
	
	\begin{align}
p(\boldsymbol{\gamma}, r)=\{(&k_{1}, \ldots, k_{|\gamma|} ; {\boldsymbol{\ell}}_{1}, \ldots, {\boldsymbol{\ell}}_{|\gamma|}) : \text { for some } 1 \leq s \leq |\gamma|,
\\ \notag
& k_{i}=0 \text { and } {\boldsymbol{\ell}}_{i}=\mathbf{0} \text { for } 1 \leq i \leq |\gamma|-s ; k_{i}>0 \text { for } |\gamma|-s+1 \leq i \leq |\gamma|;
\\ \notag
& \text{and }\boldsymbol 0\prec {\boldsymbol{\ell}}_{n-s+1}\prec \cdots \prec  {\boldsymbol{\ell}}_{n} \text{ are such that} \\ \notag
& \qquad \qquad \quad  \sum_{i=1}^{n} k_{i}=r, \sum_{i=1}^{n} k_{i} {\boldsymbol{\ell}}_{i}=\boldsymbol{\gamma} \}.
\end{align}
 	\end{lem}

\subsection{Classical solutions and mild solutions}

In this subsection, we give the definitions of classical and mild solutions for \eqref{fde}.
 \begin{defn}[Classical solution]\label{classical-solution}
  Suppose $u_0, u_1 \in C(\mathbb R^n)$ and $f\in C([0,\infty)\times \mathbb R^n)$. Then a function $u(t,x)\in C([0,\infty)\times \mathbb R^n) $ is called a classical solution of the Cauchy problem \eqref{fde} if
\begin{enumerate}
	\item $\Delta u(t,x)$ defines a continuous function of $x$ for every $t>0$,
	\item for every $x\in \mathbb R^n$, the fractional integral $I_t^{2-\alpha}u'$, where $u'=\frac{\partial u}{\partial t}$,
	\[I_t^{2-\alpha}u'(t,x)=\frac{1}{\Gamma(2-\alpha)}\int_0^t (t-\tau)^{1-\alpha}\frac{\partial u}{\partial t} (\tau, x) \dd{\tau},\]
		is continuously differentiable with respect to $t>0$, and
	\item $u(t,x)$ satisfies the equation \eqref{fde} for every $(t,x)\in (0, \infty)\times \mathbb R^n$ and the initial conditions in \eqref{fde} for every $x\in \mathbb R^n$.
\end{enumerate}

\end{defn}

Recall that $E_{\alpha, \beta}(z)$ is the Mittag-Leffler function which is given in Definition \ref{definition-of-ml}.
Define
\begin{align*}
  \widehat{S}_{\alpha}(t,\xi)&=E_{\alpha, 1}(i^{-\alpha}t^\alpha |\xi|^2),\\
 	\widehat{Q}_{\alpha}(t,\xi)&=t E_{\alpha, 2}(i^{-\alpha}t^\alpha |\xi|^2),\\
 	\widehat{P}_{\alpha}(t,\xi)&= i^{-\alpha}t^{\alpha-1}E_{\alpha, \alpha}(i^{-\alpha}t^\alpha |\xi|^2),
 \end{align*}
and $S_{\alpha}(t,x), Q_{\alpha}(t,x), P_{\alpha}(t,x)$ be the inverse Fourier transform of $\widehat{S}_{\alpha}(t,\xi), \widehat{Q}_{\alpha}(t,\xi),\widehat{P}_{\alpha}(t,\xi)$ respectively.
 We write $S_\alpha(x) = S_\alpha(1,x)$, $Q_\alpha(x) = Q_\alpha(1,x)$, and $P_\alpha(x) = P_\alpha(1,x)$for short.

When the initial data is sufficiently regular, we can use the Laplace and Fourier transforms to write the solution as the sum of three terms,
 \begin{align}\label{mildsolution}
 	u(t,x)&=\int_{\mathbb R^n}S_{\alpha}(t,x-y)u_0(y)\dd{y}+\int_{\mathbb R^n}S_{\alpha}(t,x-y)u_1(y)\dd{y}+\int_{0}^{t}\int_{\mathbb R^n}P_\alpha(t-\tau, x-y )f(\tau, y)\dd{y}\dd{\tau} \nonumber\\
 	&=\mathcal F^{-1}(\widehat{S}_{\alpha}(t,\xi) \widehat u_0(\xi))+\mathcal F^{-1}(\widehat{Q}_{\alpha}(t,\xi) \widehat u_1(\xi))+
 	\mathcal F^{-1}\left(\int_0^t\widehat{P}_{\alpha}(t-\tau,\xi))\widehat f(\tau,\xi)\dd{\tau}\right)\nonumber\\
 	&=:u_0(t,x)+u_1(t,x)+u_f(t,x).
 \end{align}

  \begin{defn} [Mild solution] Let $u_0, u_1$ and $f$ be measurable functions on $\mathbb R^n$ and $[0,\infty)\times \mathbb R^n$, respectively. Then the function $u$ defined by
\eqref{mildsolution} is called the mild solution of the Cauchy problem \eqref{fde} whenever the integrals in \eqref{mildsolution} are well defined.
\end{defn}

 \begin{rem}
 If $u_0, u_1 \in \mathcal S$ and $f\in \mathcal S$ for each $t>0$, then the mild solution is a classical solution.	 \end{rem}
 Using \eqref{asymptotic-of-ml-infity}, as $t|\xi|^{\frac{2}{\alpha}}\to \infty$, we have the following results:
 \begin{align}
 \widehat{S}_{\alpha}(t,\xi)&=\frac{1}{\alpha}e^{-it|\xi|^{\frac{2}{\alpha}}}-\sum_{k=1}^m\frac{i^{\alpha k}t^{-\alpha k}}{\Gamma(
 1-\alpha k )|\xi|^{2k}}+o\left(\frac{t^{-\alpha m}}{|\xi|^{2m}}\right),\\
\widehat{Q}_{\alpha}(t,\xi) &=\frac{1}{\alpha} i|\xi|^{-\frac{2}{\alpha}}e^{-it|\xi|^{\frac{2}{\alpha}}}-\sum_{k=1}^m\frac{i^{\alpha k}t^{1-\alpha k}}{\Gamma(2-\alpha k )|\xi|^{2k}}+o\left(\frac{t^{1-\alpha m}}{|\xi|^{2m}}\right), \  \text{and}\\
	\widehat{P}_{\alpha}(t,\xi)&=\frac{1}{\alpha} i^{-1}|\xi|^{-\frac{2(\alpha-1)}{\alpha}}e^{-it|\xi|^{\frac{2}{\alpha}}}-\sum_{k=2}^m\frac{i^{\alpha k-\alpha}t^{\alpha-1-\alpha k}}{\Gamma(\alpha-\alpha k )|\xi|^{2k}}+o\left(\frac{t^{\alpha-1-\alpha m}}{|\xi|^{2m}}\right).
\end{align} It is obvious that $\widehat{S}_{\alpha}(t,\xi), \widehat{Q}_{\alpha}(t,\xi),
\widehat{P}_{\alpha}(t,\xi) $ are oscillating as $|\xi| \to \infty$ for any fixed $t>0$.

  \begin{rem} It is easy to check that $\widehat{S}_{\alpha}(t,\xi) =D_t^1 \widehat{Q}_{\alpha}(t,\xi)$ and $\widehat{P}_{\alpha}(t,\xi)=D_t^{2-\alpha} \widehat{Q}_{\alpha}(t,\xi)$.

 \end{rem}
 \section{Asymptotic properties of the kernels}
In this section, we give the asymptotics for the oscillatory kernels and its Bessel potentials for $t=1$. In addition, we derive some upper bounds for the the Bessel potentials of the kernels for any $t>0$.

First, we consider the case when $t=1$. Recall that, for any $\theta\geq 0$,
\begin{align*}
 J^\theta S_\alpha(x)&=\frac{1}{(2\pi)^{n/2}}\int_{\mathbb R^n} e^{ ix\cdot\xi}(1+|\xi|^2)^{-\theta/2} E_{\alpha,1} (i^{-\alpha } |\xi|^2)\dd {\xi},\\
    J^\theta Q_\alpha(x)&=\frac{1}{(2\pi)^{n/2}}\int_{\mathbb R^n} e^{ ix\cdot\xi}(1+|\xi|^2)^{-\theta/2} E_{\alpha,2} (i^{-\alpha } |\xi|^2)\dd {\xi}, \ \text{and}\\
  J^\theta P_\alpha(x)&=\frac{1}{(2\pi)^{n/2}}\int_{\mathbb R^n} e^{ ix\cdot\xi} (1+|\xi|^2)^{-\theta/2}E_{\alpha,\alpha} (i^{-\alpha } |\xi|^2)\dd {\xi}.	
\end{align*}
Then we have the following lemma.
\begin{lem} \label{asymptotic-behavior-of-kernel}
For any $\theta\geq 0$.  Then $J^\theta S_\alpha(x), J^\theta Q_\alpha(x)$ and $ J^\theta P_\alpha(x)$ are all smooth functions throughout $\mathbb R^n \backslash \{0\}$ with the following asymptotic behavior as $|x|\to\infty $,
\begin{align}
J^\theta S_\alpha(x)
&
= A |x|^{\frac{n\alpha-n-\theta \alpha}{2-\alpha}}e^{-i B |x|^\frac{2}{\alpha-2}}+o(|x|^{\frac{n\alpha-n-\theta \alpha}{2-\alpha}}),
\\
J^\theta Q_\alpha (x)
&
= A |x|^{\frac{n\alpha-n-2-\alpha \theta}{2-\alpha}}e^{-i B |x|^\frac{2}{\alpha-2}}+o(|x|^{\frac{n\alpha-n-2-\alpha \theta}{2-\alpha}}),\ \text{and}
\\
J^\theta P_\alpha (x)
&
= A |x|^{\frac{n\alpha+2-n-2\alpha-\alpha \theta}{2-\alpha}}e^{-i B |x|^\frac{2}{\alpha-2}}+o(|x|^{\frac{n\alpha+2-n-2\alpha-\alpha \theta}{2-\alpha}}) ,
\end{align}
where $A$ and $B$ are some constants (which may differ from line to line) which only depend on $\alpha, \theta$ and $n$.
They have the following asymptotic behavior as $|x|\to 0$,
\begin{align}
|J^\theta S_\alpha(x)|, |J^\theta Q_\alpha(x)|&\sim
\begin{cases}C, \quad &\text{when}\quad \theta>n-2,\\
C \log\frac{1}{|x|}, \quad &\text{when} \quad \theta=n-2,\\
	C |x|^{2+\theta-n}+o(|x|^{2+\theta-n}), \quad &\text{when}\quad \theta<n-2,
\end{cases}
\end{align}
and \begin{align}
|J^\theta P_\alpha(x)|&\sim
\begin{cases}C, \quad &\text{when}\quad \theta>n-4,\\
C \log\frac{1}{|x|}, \quad &\text{when} \quad \theta=n-4,\\
	C |x|^{4+\theta-n}+o(|x|^{4+\theta-n}), \quad &\text{when}\quad \theta<n-4.
\end{cases}
\end{align}

In particular, define the constants
 \begin{align}\label{index number}
 \theta_{S_\alpha}=\frac{n}{\alpha},\  \theta_{Q_\alpha}=\frac{n}{\alpha}-\frac{2}{\alpha} \ \text{and}\ \theta_{P_{\alpha}}=\frac{n}{\alpha}+\frac{2}{\alpha}-2.	
 \end{align}
 Then
\begin{enumerate}
	\item if $\theta>\theta_{S_{\alpha}}$, we have $J^{\theta}S_{\alpha}(x)\in L^1(\mathbb R^n)$;
	\item if $\theta>\theta_{Q_{\alpha}}$, we have $J^{\theta}Q_{\alpha}(x)\in L^1(\mathbb R^n)$; and
    \item if $\theta>\theta_{P_{\alpha}}$, we have $J^{\theta}P_{\alpha}(x)\in L^1(\mathbb R^n)$.
\end{enumerate}

\end{lem}

\begin{proof} Since the proofs for the three kernels are similar, we only give the details for $S_\alpha(x)$.

	Denote by $\Phi(\xi)$ a fixed smooth function that is supported in the unit ball centered at the origin. Recall that $\widehat{J^\theta S}_\alpha(\xi)=(1+|\xi|^2)^{-\theta/2}E_{\alpha,1}(i^{-\alpha }|\xi|^2)$ is locally integrable, thus it belongs to $\mathcal S'(\mathbb R^n)$. Then we can write $J^\theta S_\alpha(x)$ as follows (interpreted in the sense of distributions),
	\begin{align*}
		J^\theta S_\alpha(x)=&\frac{1}{(2\pi)^{n/2}}\int_{\mathbb R^n} e^{ ix\cdot\xi}(1+|\xi|^2)^{-\theta/2} E_{\alpha,1} (i^{-\alpha } |\xi|^2)\dd {\xi}\\
		=&\frac{1}{(2\pi)^{n/2}}\int_{\mathbb R^n} e^{ ix\cdot\xi}(1+|\xi|^2)^{-\theta/2} E_{\alpha,1} (i^{-\alpha } |\xi|^2)\Phi(\xi)\dd {\xi}\\
		&+\frac{1}{(2\pi)^{n/2}}\int_{\mathbb R^n} e^{ ix\cdot\xi}(1+|\xi|^2)^{-\theta/2} E_{\alpha,1} (i^{-\alpha } |\xi|^2)(1-\Phi(\xi))\dd {\xi}\\
		=: &S_{\alpha, \theta,1}(x)+S_{\alpha, \theta, 2}(x).
	\end{align*}	
	Since $(1+|\xi|^2)^{-\theta/2} E_{\alpha,1} (i^{-\alpha } |\xi|^2)\Phi(\xi) \in C_c^{\infty} (\mathbb R^n)$, we have that $S_{\alpha,\theta,1}(x)$ is a Schwartz function. In particular, for any integer $N\geq 1$,
	\begin{align}\label{asymptotic of part 1}
		|S_{\alpha,\theta,1}(x)| \lesssim
		\begin{cases}
			1 &\quad \text {as} \quad |x|\to 0,\\
			 |x|^{-N}, &\quad \text {as} \quad |x| \to \infty.
		\end{cases}
	\end{align}
	 For the second term, we rewrite it as
	 \begin{align*}
	 	S_{\alpha,\theta, 2}(x)&=\frac{1}{(2\pi)^{n/2}}\int_{\mathbb R^n} e^{ ix\cdot\xi} (1+|\xi|^2)^{-\theta/2} E_{\alpha,1} (i^{-\alpha } |\xi|^2)(1-\Phi(\xi))\dd{\xi}\\
	 		 &=\frac{1}{(2\pi)^{n/2}}\int_{\mathbb R^n} e^{ix\cdot\xi}(1+|\xi|^2)^{-\theta/2}\frac{1}{\alpha} e^{-i|\xi|^{\frac{2}{\alpha}}}(1-\Phi(\xi))\dd{\xi}\\
	 	&+\frac{1}{(2\pi)^{n/2}}\int_{\mathbb R^n} e^{ ix\cdot\xi}(1+|\xi|^2)^{-\theta/2} (E_{\alpha,1} (i^{-\alpha } |\xi|^2)-\frac{1}{\alpha}e^{-i|\xi|^{\frac{2}{\alpha}}})(1-\Phi(\xi))\dd{\xi}\\
	 	&=:S_{\alpha,\theta, 21}(x)+S_{\alpha,\theta,22}(x).
	 \end{align*}
	 By Corollary \ref{equivalence of Riesz and Potential}, $S_{\alpha,\theta, 2}(x)$ is a smooth function throughout $\mathbb R^n$,
	 and
	 \begin{align}\label{asymptotic of part 2}
		|S_{\alpha,\theta,21}(x)|\sim
		\begin{cases}
			C, &\quad \text {as} \quad |x|\to 0,\\
			A |x|^{\frac{n\alpha-n-\theta\alpha}{2-\alpha}}e^{-i B |x|^\frac{2}{\alpha-2}}+o(|x|^{\frac{n\alpha-n-\theta \alpha}{2-\alpha}}), &\quad \text {as} \quad |x|\to \infty.
		\end{cases}
	\end{align}

 Now we deal with $S_{\alpha,\theta,22}(x)$. Define $r_{\alpha,\theta}(\xi)$ as the following term which appears in the integrand of $S_{\alpha,\theta,22}(x)$,
	\begin{align*}	
	r_{\alpha,\theta}(\xi):=(1+|\xi|^2)^{-\theta/2}(E_{\alpha,1} (i^{-\alpha } |\xi|^2)-\frac{1}{\alpha}e^{-i|\xi|^{\frac{2}{\alpha}}})(1-\Phi(\xi)). \end{align*}
 Employing the asymptotic formula \eqref{asymptotic-of-ml-infity}, we have for any integer $M\geq 1$,
\begin{align*}
	 r_{\alpha,\theta}(\xi)=\left(\sum_{k=1}^M\frac{C_k}{|\xi|^{2k +\theta }}+o\left(\frac{1}{|\xi|^{2k+\theta}}\right)\right)(1-\Phi(\xi))).
\end{align*}

 Using integration by parts, it is easy to see for any integer $N\geq 1$
 \begin{align}\label{asymptotic of part 3}
 	|S_{\alpha,\theta,22}(x)| \lesssim |x|^{-N}, \quad \text {as} \quad x\to \infty.
 \end{align}
It remains to consider the case when $|x|\to 0$. Using similar estimates for Lemma 3.1 in \cite{Su}, we have the following expansion for $S_{\alpha,\theta,22}(x)$,
	\begin{enumerate}
	\item  If $n>\theta+2$, for some positive integer $M$,
		\[S_{\alpha,\theta,22}(x)=\sum_{k=1}^{M}C_k|x|^{-n+2 k+\theta}+W({x}), \quad \text {as}\quad |x|\to 0,\]
	where $W \in L^{\infty}(\mathbb R^n)$.
	\item  If $n=\theta+2$, for some positive integer $M$,
	\[S_{\alpha,\theta,22}(x)=\sum_{k=1}^{M}C_k|x|^{-n+2 k+\theta}+C W_1(x)+W({x}),\]
	where $W \in L^{\infty}(\mathbb R^n)$, and $W_1(x)\sim \ln|x|$, as $|x|\rightarrow 0$.
	\item If $n<\theta+2$, then $r_{\alpha, \theta}(\xi)$ is integrable. Thus $S_{\alpha,\theta,22}(x)$ is bounded as $|x|\rightarrow 0$.
	\end{enumerate}
	
	In particular, we have
	\begin{align} \label{asymptotic of part 3'}
	| S_{\alpha,\theta,22}(x)|\sim\begin{cases}
	C, \quad &\text{when} \quad n< 2+\theta,\\
	\log \frac{1}{|x|}, \quad &\text{when} \quad n=2+\theta,\\
		|x|^{2+\theta-n},\quad &\text{when} \quad n> 2+\theta,\\
		\end{cases} \quad\text{as} \quad |x|\to 0. \end{align}

Combining \eqref{asymptotic of part 1}, \eqref{asymptotic of part 2}, \eqref{asymptotic of part 3} and \eqref{asymptotic of part 3'}, gives the asymptotic behavior for $J^\theta S_{\alpha}(x)$. Notice that, when $\theta>\frac{n}{\alpha}$, we have $\frac{n\alpha-n-\theta\alpha}{2-\alpha}<-n$,  which implies $J^\theta S_\alpha\in L^1(\mathbb R^n)$.

As for $Q_{\alpha}(x)$ and $P_{\alpha}(x)$, we can derive the desired results using the same method.
\end{proof}

Now we compare the properties of our kernels with some known results. Here we take $\theta=0$ and $S_\alpha(x)$ to illustrate our ideas.
 From Lemma $\ref{asymptotic-behavior-of-kernel}$,
$S_\alpha(x)$ has the following asymptotic behavior as $|x|\to\infty $,
\begin{align}\label{asy at infty}
 S_\alpha(x)
= A |x|^{\frac{n\alpha-n}{2-\alpha}}e^{-i B |x|^\frac{2}{\alpha-2}}+o(|x|^{\frac{n\alpha-n}{2-\alpha}}),
\end{align}
and have the following asymptotic behavior as $|x|\to 0$,
\begin{align}\label{asy at 0}
|S_\alpha(x)| &\sim
\begin{cases}C, \quad &\text{when}\quad n<2,\\
C \log\frac{1}{|x|}, \quad &\text{when} \quad n=2,\\
	C |x|^{2-n}+o(|x|^{2-n}), \quad &\text{when}\quad
	n>2.
\end{cases}
\end{align}
Let us denote $S^{\text{h}}_\alpha(x)=\mathcal F^{-1}(E_{\alpha,1}(-|\xi|^2))$, which is the fundamental solution for the fractional diffusion-wave equation \eqref{dwe}.
According to the Theorem 2.1 and Theorem 2.2 in \cite{Kim2}, $S^{\text{h}}_\alpha(x)$ has the following asymptotic behavior as $|x|\to\infty $,
\begin{align}\label{heat kernel x big}
| S^{\text{h}}_\alpha(x)|\lesssim |x|^{-n}e^{-c|x|^{\frac{2}{2-\alpha}}},
\end{align}
where $c>0$ is a constant, and has the following asymptotic behavior as $|x|\to 0$,
\begin{align}\label{heat kernel x small}
|S^{\text{h}}_\alpha(x)| &\sim
\begin{cases}C, \quad &\text{when}\quad n<2,\\
C \log\frac{1}{|x|}, \quad &\text{when} \quad n=2,\\
	C |x|^{2-n}+o(|x|^{2-n}), \quad &\text{when}\quad
	n>2.
\end{cases}
\end{align}

\begin{rem} Suppose that $1<\alpha<2$. As $|x| \to 0$, according to \eqref{asy at 0} and \eqref{heat kernel x small},
$S^{\text{h}}_\alpha(x)$ and $S_\alpha(x)$ have the same behavior and they are both in $L^1_{loc}(\mathbb R^n)$.
As $|x|\to \infty$, due to \eqref{heat kernel x big}, $S^{\text{h}}_\alpha(x)$ has exponential decay. In contrast, by \eqref{asy at infty}, the main term for $S_\alpha(x)$ is an increasing polynomial multiplied by an oscillatory term, which causes the solution for \eqref{fde} to lose regularity.
 \end{rem}

 \begin{rem} We can also compare $S_{\alpha}(x)$ with $K^{-}_{0, \frac{2}{\alpha}}(x)$.
 It is easy to see that they have the same asymptotic behavior as $|x|\to \infty$; However, they behaves differently as $|x|\to 0$, since the former one may has singularity at $|x|=0$ when the dimension is large. This difference is caused by the reminder term $r(\xi)$ of $E_{\alpha,1}(i^{\alpha}|\xi|^2)$ as $|\xi|\to \infty$.
\end{rem}

In the next lemma, we give some upper bounds for the Bessel potentials of the kernels for any arbitrary $t>0$. Since $(1+|\xi|^2)^{-\theta/2}$ is not homogeneous, the following results cannot be obtained by scaling.
Recall that $\theta_{S_\alpha},\  \theta_{Q_\alpha} \ \text{and}\ \theta_{P_{\alpha}}$ are constants defined in \eqref{index number}.

\begin{lem}\label{Bessel potential estimate}
 For the Bessel potentials of the kernels $S_{\alpha}(t,x), Q_{\alpha}(t,x)$ and $P_{\alpha}(t,x)$ with arbitrary $t>0$, we have the following upper bounds:

\begin{enumerate}
	\item \begin{enumerate}
	\item
	If $\theta>\theta_{S_\alpha}$ and $0<t<1$, then
	\begin{align}\label{estimate for S for small t}
		|J^{\theta}S_{\alpha}(t,x)|\leq \begin{cases}  C\max\{1,|x|^{\theta-n}\}, & |x| \leq t^{\frac{\alpha}{2}},\\
	C\max\{1, t^{{\frac{\alpha}{2}(\theta-n)}}\}, &   t^{\frac{\alpha}{2}}\leq |x| \leq 1,
	\\C |x|^{-n-\sigma}, & |x|\geq 1,
\end{cases}
	\end{align}
where $\sigma=\frac{\theta \alpha-n}{2-\alpha} >0$. In particular,
$\int_{R^n\setminus B(0,1)} |J^{\theta}S_{\alpha}(t,x)|\dd{x}$ is bounded uniformly in $0<t< 1$, where $B(0,1)$ is the unit ball of $\mathbb R^n$ centered at the origin.
\item If $\theta>\theta_{S_\alpha}$ and $t\geq 1$, then
	\begin{align} \label{estimate for S for big t}
		|J^{\theta}S_{\alpha}(t,x)|\leq \begin{cases}  C\max\{1,|x|^{\theta-n}\}, & |x| \leq t^{\frac{\alpha}{2}},\\
	C t^{\frac{\alpha}{2}(\theta+\sigma)} |x|^{-n-\sigma}, & |x|\geq t^{\frac{\alpha}{2}},
\end{cases}
	\end{align}
	where $\sigma =\frac{\theta \alpha-n}{2-\alpha} >0$.
\end{enumerate}
\item \begin{enumerate}
\item For $\theta>\theta_{Q_\alpha}$ and $0<t<1$, then
\begin{align}
|J^{\theta}Q_{\alpha}(t,x)|\leq \begin{cases}  C\max\{1,|x|^{\theta-n+\frac{2}{\alpha}}\}, & |x| \leq t^{\frac{\alpha}{2}},\\
	C\max\{1, t^{{\frac{\alpha}{2}(\theta-n+\frac{2}{\alpha} )}}\}, &   t^{\frac{\alpha}{2}}\leq |x| \leq 1,\\
	C|x|^{-n-\sigma}, & |x|\geq 1,
\end{cases}\end{align}
where $\sigma=\frac{\theta \alpha+2-n}{2-\alpha} >0$. In particular, $\int_{R^n\setminus B(0,1)} |J^{\theta}Q_{\alpha}(t,x)|\dd{x}$ is bounded uniformly in $0<t< 1$.
\item For $\theta>\theta_{Q_\alpha}$ and $t\geq 1$, then
\begin{align}|J^{\theta}Q_{\alpha}(t,x)|\leq \begin{cases}  C\max\{1,|x|^{\theta-n+\frac{2}{\alpha}}\}, & |x| \leq t^{\frac{\alpha}{2}},
	\\Ct^{\frac{\alpha}{2}(\theta+\sigma)}|x|^{-n-\sigma}, & |x|\geq t^{\frac{\alpha}{2}},
\end{cases}\end{align}
where $\sigma=\frac{\theta \alpha+2-n}{2-\alpha} >0$.
\end{enumerate}
\item
 \begin{enumerate}
\item
 For $\theta>\theta_{P_\alpha}$ and $0<t<1$, then
\begin{align}
|J^{\theta}P_{\alpha}(t,x)|\leq \begin{cases}  C\max\{1,|x|^{\theta-n+2-\frac{2}{\alpha}}\}, & |x| \leq t^{\frac{\alpha}{2}},\\
	C\max\{1, t^{{\frac{\alpha}{2}(\theta-n+2-\frac{2}{\alpha})}}\}, &   t^{\frac{\alpha}{2}}\leq |x| \leq 1,
	\\C |x|^{-n-\sigma}, & |x|\geq 1,
\end{cases}\end{align}
where $\sigma=\frac{\theta \alpha+2\alpha-2-n}{2-\alpha} >0$. In particular, $\int_{R^n\setminus B(0,1)} |J^{\theta} P_{\alpha}(t,x)|\dd{x}$ is bounded uniformly in $0<t< 1$.
\item  For $\theta>\theta_{P_\alpha}$ and $t\geq 1$, then
\begin{align} |J^{\theta}P_{\alpha}(t,x)|\leq \begin{cases}  C\max\{1,|x|^{\theta-n+2-\frac{2}{\alpha}}\}, & |x| \leq t^{\frac{\alpha}{2}},\\
C t^{\frac{\alpha}{2}(\theta+\sigma)}|x|^{-n-\sigma}, & |x|\geq t^{\frac{\alpha}{2}},
\end{cases}\end{align}
where $\sigma=\frac{\theta \alpha+2\alpha-2-n}{2-\alpha} >0$.
\end{enumerate}
\end{enumerate}
\end{lem}

\begin{proof}
We only prove the estimates for $J^{\theta}S_{\alpha}(t,x)$, because the estimates for the other two terms are similarly derived. Performing a  change of variables, we can write
	\begin{align*}
	J^{\theta}S_{\alpha}(t,x)&=\frac{1}{(2\pi)^{n/2}}\int_{\mathbb R^n}e^{ ix\cdot\xi}(1+|\xi|^2)^{-\frac{\theta}{2}} E_{\alpha,1}(i^{-\alpha} t^\alpha|\xi|^2)\dd{\xi}\\
	&=\frac{1}{(2\pi)^{n/2}}t^{-\alpha n/2}\int_{\mathbb R^n}e^{ i x t^{-\frac{\alpha }{2}} \cdot\xi}(1+t^{-\alpha}|\xi|^2)^{-\frac{\theta}{2}} E_{\alpha,1}(i^{-\alpha} |\xi|^2)\dd{\xi}\\
   &=\frac{1}{(2\pi)^{n/2}}t^{-\alpha n/2}\int_{\mathbb R^n}e^{ i x t^{-\frac{\alpha }{2}} \cdot\xi}(1+t^{-\alpha}|\xi|^2)^{-\frac{\theta}{2}} E_{\alpha,1}(i^{-\alpha} |\xi|^2)\Phi(\xi)\dd{\xi}\\
   &+\frac{1}{(2\pi)^{n/2}}t^{-\alpha n/2}\int_{\mathbb R^n}e^{ i xt^{-\frac{\alpha }{2}} \cdot\xi}(1+t^{-\alpha}|\xi|^2)^{-\frac{\theta}{2}} E_{\alpha,1}(i^{-\alpha} |\xi|^2)(1-\Phi(\xi))\dd{\xi}\\
   &:= I_1(t,x)+I_2(t,x).
\end{align*}
{\bf{Case $(a)$}:} when $0<t<1$.\\
{\bf{Step 1}: } First, we claim that, for any integral $N$, the following estimate is valid,
 \begin{align}\label{claim for part 1}
  	|I_1(t,x)| \lesssim \min\left\{\max \{1, t^{\alpha/2(\theta-n)}\}, \frac{t^{\alpha/2 (\theta+N-n)}+1}{|x|^{N}}\right\}.
  \end{align}
Then, we have for $N$ large enough,
\begin{align}\label{uniform estimate part 1 higher}
	|I_{1}(t,x)|\lesssim
	\begin{cases}
	\max\{ 1, |x|^{\theta-n}\}, & |x|\leq t^{\alpha/2},\\
	\max\{ 1, t^{\frac{\alpha}{2}(\theta-n)}\} &t^{\alpha/2}\leq |x|\leq 1,\\
	\frac{1}{|x|^{{N}}},  & |x|\geq 1.
	\end{cases}
\end{align}
The claim \eqref{claim for part 1} is valid, because on the one hand, by spliting the integral into two parts, i.e. $|\xi|\leq  t^{\frac{\alpha}{2}}$ and $|\xi|\geq t^{\frac{\alpha}{2}}$, we have the following trivial bound,
\begin{align}\label{trivial estimate 1}
	|I_1(t,x)|&\lesssim t^{-\frac{\alpha}{2}n} \int_{\mathbb R^n}(1+t^{-\alpha}|\xi|^2)^{-\frac{\theta}{2}} \Phi(\xi)\dd{\xi}\nonumber \\
	&\lesssim t^{-\frac{\alpha}{2}n}\int_{|\xi|\leq t^{\frac{\alpha}{2}}} \Phi(\xi)\dd{\xi}+  t^{-\frac{\alpha}{2}n}\int_{t^{\frac{\alpha}{2}}\leq |\xi|\leq 1 }t^{\frac{\alpha}{2}\theta}|\xi| ^{-\theta} \Phi(\xi)\dd{\xi}\nonumber\\
	&\lesssim t^{-\frac{\alpha}{2}n} t^{-\frac{\alpha}{2}n}+ (1+t^{\frac{\alpha}{2}(\theta-n)})\nonumber\\
	&\lesssim
	\max \{1,t^{\frac{\alpha}{2}(\theta-n) }\}.
	\end{align}
On the other hand, define
\begin{align*}
	\mathcal D u=-i \sum_{j=1}^n \frac{x_jt^{-\alpha/2}}{|x t^{-\alpha/2}|}\partial_j u
\end{align*}
then it is easy to see that its adjoint operator is $^{T}\mathcal D= -\mathcal D$ and
\begin{align*}
	\mathcal D e^{ i x t^{-\frac{\alpha }{2}} \cdot\xi}=|x t^{-\alpha/2}|e^{ i x t^{-\frac{\alpha }{2}} \cdot\xi}.
\end{align*}
Hence by integrating by parts $N$ times, we have
\begin{align*}
	I_1(t,x)&=\frac{1}{(2\pi)^{n/2}}t^{-\alpha n/2}\int_{\mathbb R^n}\frac{1}{|ixt^{-\alpha/2}|^N}\mathcal D^Ne^{ i xt^{-\frac{\alpha }{2}} \cdot\xi}\left[(1+t^{-\alpha}|\xi|^2)^{-\frac{\theta}{2}} E_{\alpha,1}(i^{-\alpha} |\xi|^2)\Phi(\xi)\right]\dd{\xi}
	\\
	&=\frac{1}{(2\pi)^{n/2}}t^{-\alpha n/2}\frac{1}{|-ixt^{-\alpha/2}|^N}\int_{\mathbb R^n} e^{ i xt^{-\frac{\alpha }{2}} \cdot\xi} \mathcal D^N\left[(1+t^{-\alpha}|\xi|^2)^{-\frac{\theta}{2}} E_{\alpha,1}(i^{-\alpha} |\xi|^2)\Phi(\xi)\right]\dd{\xi}
	\\
	&=\frac{1}{(2\pi)^{n/2}}t^{-\alpha n/2}\frac{1}{|-ixt^{-\alpha/2}|^N}\int_{|\xi|\leq t^{\alpha/2}}e^{ i xt^{-\frac{\alpha }{2}} \cdot\xi}\mathcal D^N\left[(1+t^{-\alpha}|\xi|^2)^{-\frac{\theta}{2}} E_{\alpha,1}(i^{-\alpha} |\xi|^2)\Phi(\xi)\right]\dd{\xi}\\
	&+\frac{1}{(2\pi)^{n/2}}t^{-\alpha n/2}\frac{1}{|-ixt^{-\alpha/2}|^N}\int_{t^{\frac{\alpha}{2}}\leq|\xi|\leq 1}e^{ i xt^{-\frac{\alpha }{2}} \cdot\xi} \mathcal D^N\left[(1+t^{-\alpha}|\xi|^2)^{-\frac{\theta}{2}} E_{\alpha,1}(i^{-\alpha} |\xi|^2)\Phi(\xi)\right]\dd{\xi}\\
	&=: I_{11}(t,x)+I_{12}(t,x).
\end{align*}
For $ I_{11}(t,x)$, since $t^{-\alpha}|\xi|^2\leq 1$, direct calculation yields
\begin{align*}
\big|\mathcal D^N \big[(1+t^{-\alpha}|\xi|^2)^{-\frac{\theta}{2}} E_{\alpha,1}(i^{-\alpha} |\xi|^2)\big]\big|\lesssim t^{-\frac{\alpha}{2}N}.	
\end{align*}
Thus,
\begin{align} \label{integration by part 1 low}
	|I_{11}(t,x)|\lesssim   \frac{t^{-\alpha n/2}}{|x|^{N}}\int_{|\xi|\leq t^{\alpha/2}} \dd{\xi}\lesssim \frac{1}{|x|^{N}}.
\end{align}
For $I_{22}(t,x)$, since $t^{-\alpha}|\xi|^2\geq 1$, we have $\mathcal D^N\left[(1+t^{-\alpha}|\xi|^2)^{-\frac{\theta}{2}}\right]\sim t^{\frac{\theta \alpha}{2}}|\xi|^{\theta-N}$. Thus,
\begin{align}\label{integration by part 1 high}
	| I_{12}(t,x)|&\lesssim \frac{t^{-\alpha n/2}}{|xt^{-\alpha/2}|^{
	N}}\int_{t^{\frac{\alpha}{2}}\leq|\xi|\leq 1}t^{\frac{\alpha}{2}\theta}|\xi|^{-\theta-N}\dd{\xi}\nonumber\\
	&\lesssim \frac{t^{\alpha/2 (\theta+N-n)}}{|x|^{N}}(1+t^{\alpha/2(n-\theta-N)})\lesssim \frac{t^{\alpha/2 (\theta+N-n)}+1}{|x|^{N}}.
\end{align}
Combining \eqref{trivial estimate 1}, \eqref{integration by part 1 low}, and \eqref{integration by part 1 high}, we obtain
\begin{align}\label{uniform estimate part 1 b}
	|I_1(t,x)|\lesssim \min\left\{\max \{1, t^{\alpha/2(\theta-n)}\} ,\frac{t^{\alpha/2 (\theta+N-n)}+1}{|x|^{N}}\right\}.
\end{align} \\
{\bf{Step 2}:} Next we deal with $I_2(t,x)$. In order to use the scaling property, let us introduce $\tilde{I}_2(t,x)$, defined as
 \[\tilde{I}_2 (t,x):=\frac{1}{(2\pi)^{n/2}}t^{-\alpha n/2}\int_{\mathbb R^n}e^{ ix t^{-\alpha/2}\cdot\xi}t^{\frac{\alpha \theta}{2}}|\xi|^{-\theta} E_{\alpha,1}(i^{-\alpha} |\xi|^2)(1-\Phi(\xi))\dd{\xi}.\]
In this case $t^{-\alpha/2}|\xi|\geq 1$, thus $I_2(t,x)$ and $\tilde{I}_2(t,x)$ has the same asymptotic as $ |xt^{-\frac{\alpha}{2}}|\to 0 $ or  $ |xt^{-\frac{\alpha}{2}}|\to \infty $.
By the estimates of \eqref{asymptotic of part 2}, \eqref{asymptotic of part 3} and \eqref{asymptotic of part 3'}, when $\theta>\frac{n}{\alpha}$, we have
\begin{align}\label{uniform estimate part 2}
	|I_2(t,x)|\sim|\tilde{I}_2(t,x)|\lesssim \begin{cases}  |x|^{\theta-n}
	& |xt^{-\frac{\alpha}{2}}|\leq 1,\\
C t^{\frac{\alpha}{2}(\theta+\sigma)}
|x|^{-n-\sigma} & |xt^{-\frac{\alpha}{2}}|\geq 1.
\end{cases}
\end{align}
where $\sigma=\frac{\theta \alpha-n}{2-\alpha}$.

In particular, when $|xt^{-\frac{\alpha}{2}}|\geq 1$ with $0<t<1$, we have
\[t^{\frac{\alpha}{2}(\theta+\sigma)}
|x|^{-n-\sigma}\leq t^{\frac{\alpha}{2}(\theta-n)},\]
 thus the following is also valid,
\begin{align}\label{uniform estimate part 2'}	
| I_2(t,x)|\sim|\tilde{I}_2(t,x)|\leq \begin{cases}  C\max\{1,|x|^{\theta-n}\}, & |x| \leq t^{\frac{\alpha}{2}},\\
	C\max\{1, t^{{\frac{\alpha}{2}(\theta-n)}}\}, &   t^{\frac{\alpha}{2}}\leq |x| \leq 1,
	\\C |x|^{-n-\sigma}, & |x|\geq 1.
\end{cases}\end{align}

With the help of \eqref{uniform estimate part 1 higher} and \eqref{uniform estimate part 2'},  when $0<t<1$, we get that if $\theta>\theta_{S_\alpha}$, then
	\[|J^{\theta}S_{\alpha}(t,x)|\leq \begin{cases}  C\max\{1,|x|^{\theta-n}\}, & |x| \leq t^{\frac{\alpha}{2}},\\
	C\max\{1, t^{{\frac{\alpha}{2}(\theta-n)}}\}, &   t^{\frac{\alpha}{2}}\leq |x| \leq 1,
	\\C  |x|^{-n-\sigma}, & |x|\geq 1.
\end{cases}\]
where $\sigma=\frac{\theta \alpha-n}{2-\alpha} >0$. In particular, if $0<t\leq 1$,  $J^{\theta}S_{\alpha}(t,x)$ is in $L^1(\mathbb R^n\setminus B(0,1))$ uniformly.\\
{\bf{Case $(b)$}:} when $t>1$.\\
It is easy to notice that the inequality \eqref{uniform estimate part 1 b} is still valid when $t>1$. Then we only need to estimate $I_2(t,x)$. We can rewrite $I_2(t,x)$ as
\begin{align*}
I_2(t,x)=&\frac{1}{(2\pi)^{n/2}}t^{-\alpha n/2}\int_{\mathbb R^n}e^{ i xt^{-\frac{\alpha }{2}} \cdot\xi}(1+t^{-\alpha}|\xi|^2)^{-\frac{\theta}{2}} E_{\alpha,1}(i^{-\alpha} |\xi|^2)(1-\Phi(\xi))\dd{\xi}\\
=&\frac{1}{(2\pi)^{n/2}}t^{-\alpha n/2}\int_{1\leq |\xi| \leq t^{\alpha/2}}e^{ i xt^{-\frac{\alpha }{2}} \cdot\xi}(1+t^{-\alpha}|\xi|^2)^{-\frac{\theta}{2}} E_{\alpha,1}(i^{-\alpha} |\xi|^2)(1-\Phi(\xi))\dd{\xi}\\
&+\frac{1}{(2\pi)^{n/2}}t^{-\alpha n/2}\int_{ |\xi| \geq t^{\alpha/2}}e^{ i xt^{-\frac{\alpha }{2}} \cdot\xi}(1+t^{-\alpha}|\xi|^2)^{-\frac{\theta}{2}} E_{\alpha,1}(i^{-\alpha} |\xi|^2)(1-\Phi(\xi))\dd{\xi}\\
=&: I_{21}+ I_{22}.
\end{align*}
When $t>1$ and $t^{-\alpha}|\xi|^2\leq 1$,  for $I_{21}$, we have
\begin{align}
	|I_{21}(t,x)|\lesssim t^{-\alpha n/2}( t^{\alpha n/2}+1)\lesssim 1.
\end{align}
As for $I_{22}(t,x)$, when $t^{-\alpha}|\xi|^2\geq 1$, $(1+t^{-\alpha}|\xi|^2)^{-\frac{\theta}{2}}\sim t^{\alpha \theta/2}|\xi|^{-\theta}$, then we have that
\begin{align*}
	|I_{22}(t,x)|\sim |\tilde I_2(t,x)|\lesssim \begin{cases}  |x|^{\theta-n}
	& |xt^{-\frac{\alpha}{2}}|\leq 1,\\
C t^{\frac{\alpha}{2}(\theta+\sigma)}
|x|^{-n-\sigma} & |xt^{-\frac{\alpha}{2}}|\geq 1.
\end{cases}
\end{align*}\\
Therefore, when $t>1$, we have \begin{align}
		|J^{\theta}S_{\alpha}(t,x)|\leq \begin{cases}  C\max\{1,|x|^{\theta-n}\}, & |x| \leq t^{\frac{\alpha}{2}},\\
	C t^{\frac{\alpha}{2}(\theta+\sigma)} |x|^{-n-\sigma}, & |x|\geq t^{\frac{\alpha}{2}}.
\end{cases}
	\end{align}
This finishes the proof.
\end{proof}

\begin{cor} For any fixed $t_0>0$, we have
\begin{enumerate}
	\item if $\theta> \theta_{S_\alpha}$, then there exist a constant $C>0$ such that for all  $t\in [\frac{t_0}{2}, \frac{3t_0}{2}]$,
	\[\|J^\theta S_\alpha(t,x)\|_{ L^{1}(\mathbb R^n)}\leq C.\]
   \item if $\theta> \theta_{Q_\alpha}$, then there exist a constant $C>0$ such that for all  $t\in [\frac{t_0}{2}, \frac{3t_0}{2}]$,	\[\|J^\theta Q_\alpha(t,x)\|_{ L^{1}(\mathbb R^n)}\leq C.\]
   \item if $\theta> \theta_{P_\alpha}$, then
hen there exist a constant $C>0$ such that for all  $t\in [\frac{t_0}{2}, \frac{3t_0}{2}]$,
	\[\|J^\theta P_\alpha(t,x)\|_{ L^{1}(\mathbb R^n)}\leq C.\]
\end{enumerate} 	
\end{cor}
\begin{proof} As in the earlier two proofs, we only prove that the statement is valid for $J^\theta S_\alpha(t,x)$, as the other cases are similarly proven. If $0<t_0\leq 2$, define
	\begin{align*}
		M_{t_0}(x)=
		\begin{cases}  C\max\{1,|x|^{\theta-n}\}, & |x| \leq \big(\frac{t_0}{2}\big)^{\frac{\alpha}{2}},\\
	C\max\{1, t_0^{{\frac{\alpha}{2}(\theta-n)}}\}, &   \big(\frac{t_0}{2}\big)^{\frac{\alpha}{2}}\leq |x| \leq 1,
	\\C |x|^{-n-\sigma}, & |x|\geq 1,
\end{cases}
	\end{align*}	
and if $t_0\geq 2$, define
	\begin{align*}
		M_{t_0}(x)=
		\begin{cases}  C\max\{1,|x|^{\theta-n}\}, & |x| \leq \big(\frac{t_0}{2}\big)^{\frac{\alpha}{2}},
			\\C  t_0^{\theta+\sigma}|x|^{-n-\sigma}, & |x|\geq \big(\frac{t_0}{2}\big)^{\frac{\alpha}{2}},
\end{cases}
\end{align*}
where $\sigma=\frac{\theta \alpha-n}{2-\alpha} >0$. In either case, $M_{t_0}(x)\in L^1(\mathbb R^n)$.	By the estimates \eqref{estimate for S for small t} and \eqref{estimate for S for big t}, for any $t\in [\frac{t_0}{2}, \frac{3t_0}{2}]$, $|J^\theta S_\alpha(t,x)|\leq M_{t_0}(x)$. Thus there exist a constant $C>O$ such that
 $\|J^\theta S_\alpha(t,x)\|_{L^{1}(\mathbb R^n)}\leq C$ for all  $t\in [\frac{t_0}{2}, \frac{3t_0}{2}]$.
\end{proof}

\section{The proof of the main theorem}
		
In this section, we prove Theorem \ref{main-result-1}. We divide our proof into three main subsections, i.e. space regularity,  time regularity and pointwise convergence to the initial data.

\subsection{Space regularity}
In this subsection, by exploring the multiplier properties for the corresponding Mittag--Leffler functions, we prove the space regularity for the mild solution.
To start with, we give some estimates for the derivatives of the oscillatory Mittag--Leffler functions.
\begin{lem}\label{derivative-lem}
	 Let $1<\alpha<2$, and assume that $\gamma$ is a multi-index nonnegative integer,
	 then
	 \begin{enumerate}
	 	\item when $|\xi|\leq 1$, $D^\gamma\widehat S_\alpha(\xi)$,  $D^\gamma\widehat Q_\alpha(\xi)$ and $D^\gamma\widehat P_\alpha(\xi)$  are bounded;
	 	\item when $|\xi|\geq 1$,	 we have
\begin{enumerate}
	\item $|D^\gamma\widehat S_\alpha(\xi)|\leq C |\xi|^{(\frac{2}{\alpha}-1)|\gamma|} $;
	\item $|D^\gamma\widehat Q_\alpha(\xi)|\leq C|\xi|^{-\frac{2}{\alpha}}|\xi|^{(\frac{2}{\alpha}-1)|\gamma|}$;
	\item $|D^\gamma\widehat P_\alpha(\xi)|\leq C|\xi|^{-\frac{2(\alpha-1)}{\alpha}}|\xi|^{(\frac{2}{\alpha}-1)|\gamma|}$.
\end{enumerate}

	 \end{enumerate}

\end{lem}
\begin{proof}
{\bf{Case 1}}: When $|\xi|\leq 1$, since the Mittag--Leffler function $E_{\alpha,1}(i^{-\alpha}|\xi|^2)$ is an entire function, so all its derivatives are bounded.\\
{\bf{Case 2}}: When $|\xi|>1$, let us denote $z=i^{-1}|\xi|^{\frac{2}{\alpha}}$. By induction, it is easy to verify that
 \begin{align}\label{induction estimate}
 	|D^{\gamma} z|\leq C_\gamma |\xi|^{\frac{2}{\alpha}-|\gamma|}.
 \end{align}
  Notice that when $1<\alpha<2$, we have $\frac{2}{\alpha}-1>0$, which implies that $|\xi|^{(\frac{2}{\alpha}-1)}\to \infty$ is an increasing term, and it is the main term compared with $|\xi|^{-1}$ as $|\xi|\to \infty$.\\
 (1) Firstly, we explain the result is valid for $\widehat S_\alpha(\xi)$.
   Let $\gamma$ be a multi-index. We rewrite $S_\alpha(\xi)$ as a function of $z$, that is
    $\widehat S_\alpha(\xi)=E_{\alpha,1}(i^{-\alpha} |\xi|^2)=E_{\alpha,1}(z^\alpha)$ with $z=i^{-1}|\xi|^{\frac{2}{\alpha}}$.
Then we calculate its derivatives by  Fa\'{a} di Bruno's formula Lemma \ref{faa di bruno}, and the main term for $D^{\gamma}E_{\alpha,1}(i^{-\alpha} |\xi|^2)$ is
    $\frac{d^{|\gamma|}}{dz^{|\gamma|}}E_{\alpha,1}(z^\alpha)D^\gamma z$ as $|\xi|\to \infty$. By \eqref{boundedness-of-mf} and \eqref{induction estimate}, we have
    \begin{align*}
    |D^\gamma\widehat S_\alpha(\xi)|\leq C |\xi|^{(\frac{2}{\alpha}-1)|\gamma|}.	
    \end{align*}\\

(2) Similarly, we can prove that
\begin{align}
|D^\gamma (zE_{\alpha,2}(i^{-\alpha}|\xi|^2))| &\leq |\xi|^{(\frac{2}{\alpha}-1)|\gamma|},\label{Q estimate} \\
	|D^\gamma (z^{-1})|&\leq C |\xi|^{-\frac{2}{\alpha}-|\gamma|}\label{Q estimate 2} .
\end{align}

By the Leibniz rule, we have
\begin{align}
	D^\gamma \widehat Q_\alpha(\xi)&= D^\gamma  E_{\alpha,2}(i^{-\alpha}|\xi|^2) = D^\gamma  (z^{-1} zE_{\alpha,2}(i^{-\alpha}|\xi|^2))\\
	&=\sum_{\eta \leq \gamma} \binom{\gamma}{\eta}D^\eta (z^{-1}) D^{\gamma-\eta}(zE_{\alpha,2}(i^{-\alpha}|\xi|^2)).
	\nonumber
\end{align}
By \eqref{Q estimate} and  \eqref{Q estimate 2}, we have
\begin{align}
	|D^\gamma \widehat Q_\alpha(\xi)|\leq C |\xi|^{-\frac{2}{\alpha}}|\xi|^{(\frac{2}{\alpha}-1)|\gamma|}.
\end{align}

(3) The method is the same for $P_\alpha (x)$, thus we omit the details.
 \end{proof}

\begin{lem} \label{multiplier-inside a ball}
Let $m$ be a smooth function and supported in the unit ball.  If \[\left| D ^ { \gamma } m ( \xi ) \right|\leq C,\]
	then $m\in \mathcal M(C^\theta, C^\theta )$.
\end{lem}
\begin{proof} According to Lemma \ref{Besov multiplier}, we only need to prove that $K_m(x):=\mathcal F^{-1}(m(\xi))\in B_{1,\infty}^0$.  Since $m$ is a Schwartz function, $K_m(x)\in \mathcal S\subset B_{1,\infty}^0$.	
\end{proof}

Now we are ready to give the multiplier proposition.
\begin{prop}
\begin{enumerate}
	\item $\widehat{S}_\alpha(\xi)$ is a Fourier multiplier from $C^{\frac{n}{\alpha}+s}(\mathbb R^n)$ to $C^{s}(\mathbb R^n)$;
	\item $\widehat{Q}_\alpha(\xi)$ is a Fourier multiplier from $C^{\frac{n}{\alpha}-\frac{2}{\alpha}+s}(\mathbb R^n)$ to $C^{s}(\mathbb R^n)$; and
	\item $\widehat{P}_\alpha(\xi)$ is a Fourier multiplier from $C^{\frac{n}{\alpha}+\frac{2}{\alpha}-2+s}(\mathbb R^n)$ to $C^{s}(\mathbb R^n)$.
	\end{enumerate}
In addition, these results are sharp, in the sense that the operators do not map into $C^{s'}(\mathbb R^n)$ for any $s'>s$.
\end{prop}

\begin{proof}
The proof for $\widehat{S}_\alpha(\xi),\widehat{Q}_\alpha(\xi),\widehat{P}_\alpha(\xi)$ are similar, thus we only give the details for $\widehat{S}_\alpha(\xi)$.

Let $\Phi$ be a smooth  function supported in the unit ball, then by Lemma \ref{derivative-lem}, we have
	\begin{equation}\label{low-frequency-derivative-estimate}
			|D^\gamma \widehat{S}_\alpha(\xi) \Phi(\xi)|\leq C,
		\end{equation}
		and
	\begin{equation}\label{high-frequency-derivative-estimate}
	\left|D^{\gamma}\left(\widehat{S}_\alpha(\xi)( 1-\Phi(\xi))\right)\right|\leq C |\xi|^{(\frac{2}{\alpha}-1)|\gamma|}.
	\end{equation}
	Due to Theorem \ref{multiplier-thm}, we know that $\widehat{S}_\alpha(\xi)\left( 1-\Phi(\xi)\right)$ is a multiplier from $C^{\frac{n}{\alpha}+s}(\mathbb R^n)$ to $C^{s}(\mathbb R^n)$. Due to Lemma \ref{multiplier-inside a ball}, we know $\widehat{S}_\alpha(\xi)\Phi(\xi)$ is a multiplier from $C^{\frac{n}{\alpha}+s}(\mathbb R^n)$ to $C^{\frac{n}{\alpha}+s}(\mathbb R^n)$. Collecting these two facts give  the desired result.

Next we illustrate that our index is sharp: assume that $\widehat{S}_\alpha(\xi)$ is a Fourier multiplier from $C^{\frac{n}{\alpha}+s}(\mathbb R^n)$ to $C^{s+\epsilon_0}(\mathbb R^n)$ for some $\epsilon_0>0$. Since $\widehat{S}_\alpha(\xi)\Phi(\xi)$ defines an operator from $C^{\frac{n}{\alpha}+s}(\mathbb R^n)$ to $C^{\frac{n}{\alpha}+s}(\mathbb R^n)$, then we deduce that $\widehat{S}_\alpha(\xi)\left( 1-\Phi(\xi)\right)$ is a multiplier from $C^{\frac{n}{\alpha}+s}(\mathbb R^n)$ to $C^{s+\epsilon_0}(\mathbb R^n)$.
 Let $R(\xi)=(\widehat{S}_\alpha(\xi)-\frac{1}{\alpha}e^{i|\xi|^{\frac{n}{\alpha}}})\left( 1-\Phi(\xi)\right)$. Then we have $R(\xi)\leq \frac{C}{1+|\xi|^2}$, so $R(\xi)$ is a Fourier multiplier from $C^{\frac{n}{\alpha}+s}(\mathbb R^n)$ to $C^{\frac{n}{\alpha}+s}(\mathbb R^n)$. Thus we get $\frac{1}{\alpha}e^{i|\xi|^{\frac{n}{\alpha}}}\left( 1-\Phi(\xi)\right)$ is
$C^{\frac{n}{\alpha}+s}(\mathbb R^n)$ to $C^{s+\epsilon_0}(\mathbb R^n)$, which is in contradiction with Theorem 4.4 of Miyachi \cite{Miyachi}.

 The results for $\widehat{Q}_\alpha(\xi)$ and $\widehat{P}_\alpha(\xi)$ are similarly proven.
\end{proof}
For Mittag-Leffler kernels, we have the following  scaling properties by changing variables,
\begin{align*}
	S_\alpha(t,x)&=t^{-\frac{\alpha}{2}n}S_\alpha(1, t^{-\frac{\alpha}{2}}x),\\
	Q_\alpha(t,x)&=t^{-\frac{\alpha}{2}n+1}Q_\alpha(1, t^{-\frac{\alpha}{2}}x),\\
	P_\alpha(t,x)&=t^{-\frac{\alpha}{2}n+\alpha-1}P_\alpha(1, t^{-\frac{\alpha}{2}}x).
\end{align*}
therefore we have proven:
if $u_0(x)\in C^{\frac{n}{\alpha}+s}(\mathbb R^n)$,  $u_1(x)\in C^{\frac{n}{\alpha}-\frac{2}{\alpha}+s}(\mathbb R^n)$  and $f(t,x )\in  C^{\frac{n+2}{\alpha}-2+s}(\mathbb R^n)$ for any $t>0$, then we have $u(t,x) \in C^{s}(\mathbb R^n)$ for every $t>0$.

 \subsection{Time regularity}
	In this subsection, we will demonstrate that, if $u_0 \in C^{\frac{n}{\alpha}+\epsilon}(\mathbb R^n)$, $u_1 \in C^{2+\frac{n}{\alpha}-\frac{2}{\alpha} +\epsilon}(\mathbb R^n)$ and$f\in C_t^{0}((0,\infty): C^{\frac{n}{\alpha}+\frac{2}{\alpha}+\epsilon}(\mathbb R^n))$, then $J^{2-\alpha}u{'}$ is continuously differentiable in time for $t>0$.

In order to do that, notice if the initial data is regular enough, then we can write  $\frac{\partial }{\partial t}I_t^{2-\alpha}u{'}$ as three terms
\[\frac{\partial}{\partial t}I_t^{2-\alpha}u{'}=\frac{\partial }{\partial t}I_t^{2-\alpha}u_{0}'(t,x)+\frac{\partial}{\partial t}I_t^{2-\alpha}u_{1}'(t,x)+\frac{\partial} {\partial t}I_t^{2-\alpha}u_{f}'(t,x),\] where
\begin{align*}
\frac{\partial }{\partial t}I_t^{2-\alpha}u_{0}'(t,x)&=\int_{\mathbb R^n}M_{\alpha}(t,x-y)u_0(y)\dd{y},\\
\frac{\partial }{\partial t}I_t^{2-\alpha}u_{1}'(t,x)&=\int_{\mathbb R^n}N_{\alpha}(t,x-y)u_1(y)\dd{y},\\
\frac{\partial} {\partial t}I_t^{2-\alpha}u_{f}'(t,x)&=f(t,x)+\int_{0}^t\int_{\mathbb R^n}L_{\alpha}(t-\tau,x-y)f(\tau, y)\dd{y}\dd{\tau},
\end{align*}
with \begin{align*}
	\widehat{M_\alpha}(t,\xi)&= i^{-\alpha} |\xi|^2 E_{\alpha,1}(i^{-\alpha}t^\alpha |\xi|^2 ),\\
	\widehat{N_\alpha}(t,\xi)&= t^{1-\alpha} E_{\alpha,2-\alpha}(i^{-\alpha}t^\alpha |\xi|^2 ),\\
	\widehat{L_\alpha}(t,\xi)&= i^{-\alpha} |\xi|^2 t^{\alpha-1} E_{\alpha,\alpha}(i^{-\alpha}t^\alpha |\xi|^2 ).
\end{align*}
Therefore if we can verify that
\[M_\alpha(t,\cdot)*u_0(x), \quad N_\alpha(t,\cdot)*u_1(x),\quad \text{and}\quad\int_{0}^{t}\int_{\mathbb R^n}L_\alpha(t-\tau, x-y )f(\tau, y)\dd{y}\dd{\tau}\]
exist for every $x\in \mathbb R^n$ and it is continuous with respect to $t>0$, then we have our desired results.

For the above kernels, we have the following  scaling properties.
\begin{align*}
	M_\alpha(t,x)&=t^{-\frac{\alpha}{2}n-\alpha}M_\alpha(1, t^{-\frac{\alpha}{2}}x),\\
	N_\alpha(t,x)&=t^{-\frac{\alpha}{2}n+1-\alpha}N_\alpha(1, t^{-\frac{\alpha}{2}}x), \quad \text{and}\\
	L_\alpha(t,x)&=t^{-\frac{\alpha}{2}n-1}L_\alpha(1, t^{-\frac{\alpha}{2}}x)
\end{align*}
Analogously,  we write $M_\alpha(x) = M_\alpha(1,x), N_\alpha(x) = N_\alpha(1,x)$ and $ L_\alpha(x) = L_\alpha(1,x)$ for short.

\begin{lem} Let $\theta>0$.  Then $J^\theta M_\alpha(x), J^\theta N_\alpha(x)$ and $ J^\theta L_\alpha(x)$ are all smooth functions throughout $\mathbb R^n \backslash \{0\}$ and have the following asymptotic behaviors as $|x|\to\infty $, there exist some constants $A$ and $B$  which depends on $\alpha, \theta$ and $n$, such that,
\begin{align}
J^\theta M_\alpha (x)
&
= A |x|^{\frac{n\alpha+2\alpha-n-\alpha \theta}{2-\alpha}}e^{i B |x|^\frac{2}{\alpha-2}}+o(|x|^{\frac{n\alpha+2\alpha-n-\alpha \theta}{2-\alpha}}),
\\
J^\theta N_\alpha (x)
&
= A |x|^{\frac{n\alpha+2(\alpha-1)-n-\alpha \theta}{2-\alpha}}e^{i B |x|^\frac{2}{\alpha-2}}+o(|x|^{\frac{n\alpha+2(\alpha-1)-n-\alpha \theta}{2-\alpha}}) ,
\\
J^\theta L_\alpha (x)
&
= A |x|^{\frac{n\alpha+2-n-\alpha \theta}{2-\alpha}}e^{i B |x|^\frac{2}{\alpha-2}}+o(|x|^{\frac{n\alpha+2-n-\alpha \theta}{2-\alpha}}) ,
\end{align}
and the following asymptotic behaviors as $|x|\to 0$, and
\begin{align}
|J^\theta M_\alpha(x)|\sim
\begin{cases}C, \quad &\text{when}\quad \theta>n,\\
C \log\frac{1}{|x|}, \quad &\text{when} \quad \theta=n,\\
	C |x|^{\theta-n}+o(|x|^{\theta-n}), \quad &\text{when}\quad \theta<n.
\end{cases}
\end{align}
\begin{align}
 |J^\theta N_\alpha(x)|, |J^\theta L_\alpha(x)|&\sim
\begin{cases}C, \quad &\text{when}\quad \theta>n-2,\\
C \log\frac{1}{|x|}, \quad &\text{when} \quad \theta=n+2,\\
	C |x|^{2+\theta-n}+o(|x|^{2+\theta-n}), \quad &\text{when}\quad \theta<n-2.
\end{cases}
\end{align}
In addition, if we denote $\theta_{{M_\alpha}}={\frac{n}{\alpha}+2}$, $\theta_{N_\alpha}={\frac{n}{\alpha}+2-\frac{2}{\alpha}}$, $ \theta_{L_\alpha}=\frac{n}{\alpha}+\frac{2}{\alpha} $, then we have
\begin{enumerate}
 \item if $\theta>\theta_{M_\alpha}$, then $J^{\theta}M_{\alpha}(x)\in L^1(\mathbb R^n)$;
 \item if $\theta>\theta_{N_\alpha}$, then $J^{\theta}N_{\alpha}(x)\in L^1(\mathbb R^n)$; and
    \item if $\theta>\theta_{L_\alpha}$, then $J^{\theta}L_{\alpha}(x)\in L^1(\mathbb R^n)$.
\end{enumerate}
\end{lem}
\begin{proof} The proof is the similar to Lemma \ref{asymptotic-behavior-of-kernel}, thus we omit the details.	
\end{proof}
\begin{lem} For any fixed $t_0>0$, we have
\begin{enumerate}
	\item if $\theta> \theta_{M_\alpha}$, then there exist a constant $C>0$ such that for all  $t\in [\frac{t_0}{2}, \frac{3t_0}{2}]$,
	\[\|J^\theta M_\alpha(t,x)\|_{ L^{1}(\mathbb R^n)}\leq C.\]
   \item if $\theta> \theta_{N_\alpha}$, then there exist a constant $C>0$ such that for all  $t\in [\frac{t_0}{2}, \frac{3t_0}{2}]$,	\[\|J^\theta N_\alpha(t,x)\|_{ L^{1}(\mathbb R^n)}\leq C.\]
   \item if $\theta> \theta_{L_\alpha}$, then
hen there exist a constant $C>0$ such that for all  $t\in [\frac{t_0}{2}, \frac{3t_0}{2}]$,
	\[\|J^\theta L_\alpha(t,x)\|_{ L^{1}(\mathbb R^n)}\leq C.\]
\end{enumerate} 	\end{lem}

\begin{proof}
	The proof is similar to Lemma \ref{Bessel potential estimate} and omit the details.
\end{proof}

\begin{prop} Assume that $u_0\in C^{\frac{n}{\alpha}+2+\epsilon}(\mathbb R^n)$, $u_1\in C^{\frac{n}{\alpha}+2-\frac{2}{\alpha}+\epsilon}(\mathbb R^n)$ and $f\in C^0_t((0,\infty); C^{\frac{n+2}{\alpha}+\epsilon}(\mathbb R^n))$, then $I_t^{2-\alpha}u'(t,x)$ is continuous differentiable with respect to $t>0$.
\begin{proof} If $u_0\in C^{\frac{n}{\alpha}+2+\epsilon}(\mathbb R^n)$, then $J^{-(\frac{n}{\alpha}+2+\epsilon)}u_0(x)\in L^{\infty}(\mathbb R^n)\cap C(\mathbb R^n)$. For any fixed $t_0>0$ $J^{\frac{n}{\alpha}+2+\epsilon}M_\alpha(t,x) \in L^1(\mathbb R^n)$ for all $t\in[t_0/2, 3t_0/2]$. By the dominated convergence theorem, we can change the limit and the integral, thus the following equality is valid,
	  $$\frac{\partial }{\partial t}I_t^{2-\alpha}u_{0}'(t,x)\big |_{t=t_0}=J^{\frac{n}{\alpha}+2+\epsilon}M_\alpha(t_0,\cdot)*J^{-\frac{n}{\alpha}+2+\epsilon}u_0(x)=M_\alpha(t_0,\cdot)*u_0(x),$$
	  thus $\frac{\partial }{\partial t}I_t^{2-\alpha}u_{0}'(t,x)$ is continuous at every $t_0>0$.
	
	   That is, for all $x\in \mathbb R^n$, $I_t^{2-\alpha}u_{0}'(t,x)$ is continuously differentiable with respect to $t>0$.
	  Similarly, $x\in \mathbb R^n$, $I_t^{2-\alpha}u_{1}'(t,x)$ and $I_t^{2-\alpha}u_{f}'(t,x)$ is continuously differentiable with respect to $t>0$. Thus the statement is valid.
\end{proof}
	
\end{prop}

{Conclusion}: If $u_0(x)\in C^{\frac{n}{\alpha}+2+\epsilon}(\mathbb R^n), u_1(x) \in  C^{\frac{n}{\alpha}+2-\frac{2}{\alpha}++\epsilon}(\mathbb R^n)$ and $f\in C^0_t((0,\infty); C^{\frac{n}{\alpha}+\frac{2}{\alpha}+\epsilon}(\mathbb R^n))$, then for each $x\in \mathbb R^n$, $I_t^{2-\alpha}u'(t,x)$ is continuously differentiable with respect to $t$.

\subsection{Pointwise convergence to the initial data}
In this  subsection, we restrict the time $t$ to $0<t<1$, since we only discuss the behavior $u(t,x)$ as $t\to 0$.
 In order to demonstrated the validity of the pointwise convegence, we give the following `almost finite speed propagation' property. Let us denote $\Phi_R(x)$ be a smooth function equal to $1$ in $B(0, R)$ and $0$ in $B(0, R+1)^c$.

%

\begin{lem}[Mismatch estimates]\label{mismatch estimate} Assume that $0<t<1$, and $\varphi(x)\in C^{\theta}(\mathbb R^n)$.
 \begin{enumerate}
 	\item If $\theta>\theta_{S_{\alpha}}$,  then for any $\delta >0$, there exist a number $R_\delta>1$, such that
 	\begin{equation}
	\big\|\chi_{B(0,1)}(x)\int_{\mathbb R^n}J^\theta S_\alpha(t,x-y)(1-\Phi_{R_\delta}(y))J^{-\theta}\varphi(y)\dd{y}\big\|_{L^\infty(\mathbb R^n)}\leq\delta \|J^{-\theta}\varphi(x)|_{L^\infty(\mathbb R^n)}.
\end{equation} 	
\item If $\theta>\theta_{Q_{\alpha}}$, then for any $\delta >0$, there exist a number $R_\delta>1$, such that
\begin{equation}
	\big\|\chi_{B(0,1)}(x)\int_{\mathbb R^n}J^\theta Q_\alpha(t,x-y)(1-\Phi_{R_\delta}(y))J^{-\theta}\varphi(y)\dd{y}\big\|_{L^\infty(\mathbb R^n)}\leq\delta \|J^{-\theta}\varphi(x)|_{L^\infty(\mathbb R^n)}.
\end{equation}
 \item If $\theta>\theta_{P_{\alpha}}$, then for any $\delta >0$, there exist a number $R_\delta>1$, such that\begin{equation}
	\big\|\chi_{B(0,1)}(x)\int_{\mathbb R^n}J^\theta P_\alpha(t,x-y)(1-\Phi_{R_\delta}(y))J^{-\theta}\varphi(y)\dd{y}\big\|_{L^\infty(\mathbb R^n)}\leq\delta \|J^{-\theta}\varphi(x)|_{L^\infty(\mathbb R^n)}.
	\end{equation}
 \end{enumerate}
 \end{lem}
\begin{proof}We give the proof for $S_\alpha (t,x)$, the others are analogous. By Lemma \ref{Bessel potential estimate}, $J^{\theta}S_{\alpha}(t,x) \in L^1(\mathbb R^n \setminus B(0,1))$ for all $0<t<1$. That is, for any $\delta >0$, there exist a number $R_\delta >1$ big enough such that for any $0<t<1$,
\begin{align*}
	\int_{|x|\geq R_\delta-1} \left|J^{\theta}S_{\alpha}(t,x)\right|\dd{x}\leq \delta.
\end{align*}
When $\varphi(x)\in C^{\theta}(\mathbb R^n)$, we have
  \begin{equation*}
  	\|((1-\Phi_{R_\delta}(x))J^{-\theta}\varphi(x) )\|_{L^\infty(\mathbb R^n)}\leq \|J^{-\theta}\varphi(x) \|_{L^\infty(\mathbb R^n)}< \infty.
  \end{equation*}
Since we have $x\in B(0,1)$, $y\in B(0, {R_\delta})^c$, then $x-y \in B(0, {R_\delta}-1)^c$, thus
\begin{align*}
	& |\chi_{B(0,1)}(x)\int_{\mathbb R^n}J^\theta S_\alpha(t,x-y)(1-\Phi_{R_\delta}(y))J^{-\theta}\varphi(y)\dd{y}|\\
	\leq &\|((1-\Phi_{R_\delta})(y)J^{-\theta}\varphi(y))\|_{L^\infty (\mathbb R^n)}\int_{|y|\geq {R_\delta}-1}\left|J^{\theta} S_{\alpha}(t,y)\right|\dd{y}\\
	\leq &\delta \|J^{-\theta}\varphi(x)\|_{L^\infty(\mathbb R^n)}.
\end{align*}
Thus we proved the lemma.
\end{proof}

\begin{prop} If $u_0 \in C^{2+\frac{n}{\alpha}+\epsilon}(\mathbb R^n)$, $u_1 \in C^{2+\frac{n}{\alpha}-\frac{2}{\alpha} +\epsilon}(\mathbb R^n)$ and $f\in C_t^{0}((0,\infty); C^{\frac{n}{\alpha}+\frac{2}{\alpha}+\epsilon}(\mathbb R^n))$, then the mild solution $u(t,x)$ pointwise convergence to the initial data, that is $\lim\limits_{t\to 0}u(t,x)=u_0(x)$ and $\lim\limits_{t\to 0}\partial_t u(t,x)=u_1(x)$.	
\end{prop}

\begin{proof}It is easy to notice that $\widehat{S}_{\alpha}(t,\xi), \widehat{Q}_{\alpha}(t,\xi), \widehat{P}_{\alpha}(t,\xi)$ are continuous function at $t=0$, and $\widehat{S}_{\alpha}(0,\xi)=1, \widehat{Q}_{\alpha}(0,\xi)=0, \widehat{P}_{\alpha}(0,\xi)=0$.\\
{\bf{Step 1:}} First we claim that if $J^{-\theta}u_0$ has compact support, the pointwise convergence is valid.
This is because if $J^{-\theta}u_0\in C^0_c(\mathbb R^n)$ with $\theta>\frac{
n}{\alpha}>\frac{n}{2}$, then $u_0\in H^{\theta}(\mathbb R^n)\hookrightarrow L^\infty(\mathbb R^n)$. Also we have
\[|\widehat u_0(t,\xi)|= |\widehat u_0(\xi)E_\alpha (i^{-\alpha}t^\alpha |\xi|^2)|\leq C |\widehat u_0(\xi)|,\]
thus, $u_0(t,x)\in H^{\theta}(\mathbb R^n)\hookrightarrow L^\infty(\mathbb R^n)$. Therefore,
\begin{align*}
	\|u_0(t,x)-u_0(x)\|_{L^\infty(\mathbb R^n)}&\leq \|u_0(t,x)-u_0(x)\|_{H^{\theta}(\mathbb R^n)}\\
	&\leq \|(E_\alpha (i^{-\alpha}t^\alpha |\xi|^2)-1)(1+4\pi|\xi|^2)^{\theta/2}\widehat u_0(\xi)\|_{L^2}\\
	&\to 0 \quad \text{as} \quad t\to 0.
\end{align*}
 In other words, if we assume that $\operatorname {supp}  (J^{-\theta}u_0)\subset B(0, R)$ for a fixed $R>0$. Then for any $\delta>0$, there exist $0<t_\delta<1$ such that for any $0<t<t_\delta$ and $x\in \mathbb R^n$, we have
\begin{align}
	|u_0(t,x)-u_0(x)|
	&=\left|\int_{\mathbb R^n} J^\theta S_{\alpha}(t, x-y) \Phi_{R}(y) J^{-\theta }u_0(y)\dd{y}-J^\theta (\Phi_R (x)J^{-\theta}u_0(x)) \right|\nonumber\\
	&\leq \frac{\delta}{2}.
\end{align}
{\bf{Step 2:}} For general $u_0\in C^{\theta}(\mathbb R^n)$ with $\theta>\frac{n}{\alpha}$, we only need to prove that for any $x\in B(0,1)$, $u_0(t,x)\to u_0(x)$, which is equivalent to $\chi_{B(0,1)} u_0(t,x)\to \chi_{B(0,1)}u_0(x)$.
For any $\delta>0$, due to Lemma \ref{mismatch estimate} and the integrability of $G^\theta(x)$, there exist a $R_\delta >1$ such that for any $0<t<1$,
\begin{align}\label{outside a ball}
	\big|\chi_{B(0,1)}(x)\int_{\mathbb R^n}J^\theta S_\alpha(t,x-y)(1-\Phi_{R_\delta}(y))J^{-\theta}u_0(y)\dd{y}\big|\leq\frac{\delta}{2},
	\end{align}
and
	\begin{align}\label{mismatch for Bessel potential}
		\left|\chi_{B(0,1)}(x) \int_{\mathbb R^n} G^\theta (x-y)(1-\Phi_{R_\delta}(y))J^{-\theta}u_0(y))\dd{y}\right|\leq \frac{\delta}{4}.
	\end{align}	
For the above $R_\delta$, since $\Phi_{R_\delta}J^{-\theta}u_0$ is compactly supported, then by the step 1, there exists $t_\delta$ such that for any $0<t<t_\delta$, we have
\begin{align}\label{inside a ball}
\left|\int_{\mathbb R^n}J^\theta S_\alpha(t,x-y)(\Phi_{R_\delta}(y)J^{-\theta}u_0(y))\dd{y}-J^\theta (\Phi_{R_\delta} (x)J^{-\theta}u_0(x))\right|\leq\frac{\delta}{4}.
\end{align}
Combing \eqref{outside a ball}, \eqref{mismatch for Bessel potential} and \eqref{inside a ball}, we have
\begin{align}
	&\big|\chi_{B(0,1)}(x)u_0(t,x) -\chi_{B(0,1)}(x)u_0(x))\big|\nonumber\\
	=&\big|\chi_{B(0,1)}(x)\int_{\mathbb R^n}J^\theta S_\alpha(t,x-y)(\Phi_{R_\delta}(y)J^{-\theta}u_0(y))\dd{y}\nonumber\\
	 &+\chi_{B(0,1)}(x)\int_{\mathbb R^n}J^\theta S_\alpha(t,x-y)\big((1-\Phi_{R_\delta}(y))J^{-\theta}u_0(y)\big)\dd{y}\nonumber\\
	&- \chi_{B(0,1)}(x) J^\theta (\Phi_{R_\delta}(x)J^{-\theta}u_0(x))-\chi_{B(0,1)}(x) J^\theta \big((1-\Phi_{R_\delta}(x))J^{-\theta}u_0(x)\big)\nonumber\big|\\
	\leq &\left|\chi_{B(0,1)}(x)\left(\int_{\mathbb R^n}J^\theta S_\alpha(t,x-y)(\Phi_{R_\delta}(y)J^{-\theta}u_0(y))\dd{y}- J^\theta (\Phi_{R_\delta}(x)J^{-\theta}u_0(x))\right)\right|\nonumber\\
	&+\left| \chi_{B(0,1)}(x)\int_{\mathbb R^n}J^\theta S_\alpha(t,x-y)\big((1-\Phi_{R_\delta}(y))J^{-\theta}u_0(y)\big)\dd{y}\right|\nonumber\\
	&+\big|\chi_{B(0,1)}(x) J^\theta \big((1-\Phi_{R_\delta}(x))J^{-\theta}u_0(x)\big)\big|\nonumber\\
	\leq & \frac{\delta}{4}+\frac{\delta}{2}+\frac{\delta}{4}=\delta.\nonumber
\end{align}
that is  $\lim\limits_{t\to 0}u_0(t,x)= u_0(x)$ for any $x\in B(0,1)$.

Similarly, we can prove that $\lim_{t\to 0}u_1(t,x)=0$ and $\lim_{t\to 0}u_f(t,x)=0$ under our assumptions. Therefore
\[\lim_{t\to 0}u(t,x)=u_0(x)\]

{\bf{Step 3:}} By direct calculation, we have
\begin{align*}
	\partial_t u(t,x)= P_{\alpha}(t,x)*u_0(x)+S_{\alpha}(t,x)*u_1(t,x)+\int_{0}^t \widehat H(t-\tau, \xi)\widehat f(\tau, \xi)\dd{\xi},
\end{align*}
where $\widehat H(t-\tau, \xi)=i^{-\alpha}t^{\alpha-2}E_{\alpha,\alpha-1}(i^{-\alpha}t^\alpha |\xi|^2)$. Notice that $\widehat H(t-\tau, \xi)$ is locally $L^1$ as a function of $t$, using the same argument in step 1, we can prove that under our assumption,
\[\lim_{t\to 0}\partial_t u(t,x)=u_1(x).\]
Thus we verifies the proposition.
\end{proof}

\section{The Schr\"odinger equation}

In this section, we discuss the H\"older regularity for the  Schr\"odinger equation and demonstrate that Theorem \ref{main-result-2} is valid.
 \begin{defn}The mild solution \eqref{se} is given by
 \begin{align}\label{sms}
 	u(t,x)&=S_1(t,\cdot)*u_0(x)+\int_{0}^{t}\int_{\mathbb R^n}S_1(t-\tau, x-y )f(\tau, y)\dd{y}\dd{\tau}\\
 	&:=u_0(t,x)+u_f(t,x),\nonumber
 \end{align}
 where $\widehat S_{1}(t,\xi)= e^{-i|\xi|^2 t}$, as long as the integral exist.
 \end{defn}

 \begin{lem} $\widehat S_{1}(\xi)$ is a multiplier from $C^{n+s}(\mathbb R^n)$ to $C^s(\mathbb R^n)$.\end{lem}
 \begin{proof} We split $\widehat S_{1}(\xi)$ into high frequency part and low frequency part. For the high frequency part, it is a multiplier from $C^{n+s}(\mathbb R^n)$ to $C^s(\mathbb R^n)$ by Theorem 4.4 in \cite{Miyachi}. And the low frequency part is  a multiplier from $C^{n+s}(\mathbb R^n)$ to $C^{s+n}(\mathbb R^n)$ by Lemma \ref{multiplier-inside a ball}. Thus we arrive at this statement.	
 \end{proof}
 	For the Schr\"odinger equation, using the similar estimates in Lemma \ref{Bessel potential estimate} for $J^\theta S_1(t,x)$ and $J^{\theta} \partial_t S_{1} (t,x)$, we can prove that when $\theta>\theta_{S_1}=n$, $J^{\theta} S_{1} (t,x) \in L^1(\mathbb R^n) $ uniformly in $0<t<1$ and when $\theta>n+2$, $J^{\theta} \partial_t S_{1} (t,x) \in L^1(\mathbb R^n) $ uniformly in $0<t< 1$. Therefore when $u_0(x)\in C^{n+2+\epsilon}(\mathbb R^n)$, $u(t,x)$ is continuous differentiable with respect to $t$ for every $x\in \mathbb R^n$ and it is the pointwise convergence to the initial data using the similar estimates for $J^\theta S_1(t,x)$. Thus we omit the details. In conclusion, Theorem \ref{main-result-2} is valid.

 Here we give another simpler proof for the pointwise convergence due to the high regularity for the initial data.
\begin{prop}If $u_0(x)\in C^{n+\epsilon}(\mathbb R^n)$ with $\epsilon>0$,  the mild solution \eqref{sms} pointwise convergent to the initial date.
\end{prop}
\begin{proof}
	By the dominated convergence theorem, the following equality is valid,
\[\lim_{t\to 0}\int_{\mathbb R^n}e^{i x \cdot \xi }e^{-i|\xi|^2t}\frac{\dd{\xi}}{(1+|\xi|^2)^{\frac{n+\epsilon}{2}}}=\int_{\mathbb R^n}e^{i x \cdot \xi }\frac{\dd{\xi}}{(1+|\xi|^2)^{\frac{n+\epsilon}{2}}}=G^{n+\epsilon}(x),
 \]
 therefore we have
 \begin{align*}
	\lim_{t\to 0}u_0(t,x)&=\lim_{t\to 0}S_1(t,\cdot)*u_0(x)\\
 	&=\lim_{t\to 0}J^{n+\epsilon}S_1(t,\cdot)*J^{-(n+\epsilon)}u_0(x)\\
 	&=G^{n+\epsilon}*J^{-(n+\epsilon)}u_0(x)=u_0(x).
 \end{align*}
  where we use the fact that $J^{n+\epsilon}S_1(t,x)$ is controlled by a integrable function pointwise uniformly for $0<t<1$ to verify the exchange of the integral and the limit.
 Similarly, we have $\lim\limits_{t\to 0}u_f(t,x)=0$. In this way, we proved the pointwise convergence.

 \end{proof}

 \noindent\textbf{Acknowledgements}.
Jiqiang Zheng was partially supported by NSFC under grant 11831004 and 11771041.

\end{document}